\documentclass[aap]{amsart}
\usepackage{a4wide}
\usepackage{amsfonts}
\usepackage{psfrag}

\usepackage{amsmath}
\usepackage{amssymb}
\usepackage{amsthm}
\usepackage{graphicx}
\newtheorem {lemma}{Lemma}[section]
\newtheorem {thm}{Theorem}
\newtheorem{proposition}[lemma]{Proposition}

\newtheorem {rem}[lemma]{Remark}
\newtheorem {cor}[lemma]{Corollary}

\newcommand{\E}{\mathbb{E}}
\newcommand{\N}{\mathbb{N}}
\newcommand{\R}{\mathbb{R}}
\newcommand{\eps}{\epsilon}
\newcommand{\1}{\mathbf{1}}
\newcommand{\ARG}{{\boldsymbol\cdot}}
\def\NU{\nu}
\def\MU{\mu}

\newcommand{\To}[1]{\,\stackrel{#1}{\longrightarrow}\,}
\newcommand{\Toi}[1]{\To{#1 \rightarrow \infty}}

\newcommand{\limN}{\Toi{N}}
\def\rX{\tilde X}
\def\bZ{{\mathbf Z}}
\def\brZ{{\tilde {\mathbf Z}}}
\def\rZ{{\tilde Z}}
\def\rI{\tilde I}
\def\mtimes{\diamond}

\renewcommand{\P}{\mathbb{P}}
\numberwithin{equation}{section}
\numberwithin{equation}{section}
\numberwithin{equation}{section}

\newcommand{\mfalls}{\quad\mbox{if \;}}

\newcommand{\ve}{\varepsilon}
\newcommand{\CA}{\mathcal{A}}
\newcommand{\etAN}{e^{t\CA^N}}
\newcommand{\CF}{\mathcal{F}}
\newcommand{\CL}{\mathcal{L}}
\newcommand{\wlimN}{\stackrel{N \rightarrow \infty}{\Longrightarrow}}
\newcommand{\sumkS}{\sum_{k=1}^N}
\begin{document}
\title[Finite System Scheme for Infinite Rate Mutually Catalytic Branching]{Finite System Scheme for Mutually Catalytic Branching with infinite branching rate}

\author{Leif D\"{o}ring}
\address{School of Business Informatics and Mathematics\\ Universit\"{a}t Mannheim\\ 68131 Mannheim\\Germany}

\thanks{}
\email{doering@uni-mannheim.de}
\author{Achim Klenke}
\address{Institut f\"{u}r Mathematik\\Johannes Gutenberg-Universit\"{a}t Mainz\\Staudingerweg 9\\55099 Mainz\\Germany}
\email{math@aklenke.de}
\author{Leonid Mytnik}
\address{Faculty of Industrial Engineering
and Management\\Technion Israel Institute of Technology\\Haifa 32000\\Israel}
\email{leonid@ie.technion.ac.il}
\thanks{A. Klenke and L. Mytnik acknowledge support by the German Israeli Foundation (GIF) with grant 1170-186.6/2011. L. D\"{o}ring acknowledges support by the Swiss national fund (SNF) with the Ambizione grant 148117.
}
\date{18 January 2017, revised version}

\begin{abstract}
For many stochastic diffusion processes with mean field interaction, convergence of the rescaled total mass processes towards a diffusion process is known.

Here we show convergence of the so-called finite system scheme for interacting jump-type processes known as mutually catalytic branching processes with infinite branching rate. Due to the  lack of second moments the rescaling of time is different from the finite rate mutually catalytic case.
The limit of rescaled total mass processes  is identified as the finite rate mutually catalytic branching diffusion. The convergence
of rescaled processes holds jointly with convergence of coordinate processes, where the latter converge at a different time scale.
\end{abstract}
\maketitle

\section{Introduction and Main Results}
\label{S1}
\subsection{The finite systems scheme}
\label{S1.1}
The finite systems scheme for interacting diffusion processes was developed by Cox and Greven \cite{CoxGreven1990} and Cox, Greven and Shiga \cite{CoxGrevenShiga1995} as a tool for a quantitative description of large, but finite, systems in terms of the equilibrium distributions of their infinite counterparts. In order to describe the idea, it is most convenient to sketch an example. In fact, we will only describe the so-called mean field finite systems scheme here. For $N\in\N$ let $S^N:=\{1,\ldots,N\}$ be a finite site space. Each site $k\in S^N$ carries a diffusion process $(X^N_t(k))_{t\geq0}$ with values in an interval $I$. Furthermore, the diffusion processes interact mutually via symmetric migration. More formally, we have the following set of stochastic differential equations (the second line being an equivalent reformulation of the first line):
\begin{equation}
\label{E1.01}
\begin{aligned}
	dX^N_t(k)&=(\mathcal A^N X^N_t)(k)\,dt+\sqrt{g(X^N_t(k))}\,dB_t(k)\\
&=\frac{1}{N}\sum_{l\in S^N} \big(X^N_t)(k)-X^N_t)(k)\big)\,dt+\sqrt{g(X^N_t(k))}\,dB_t(k),\qquad k\in S^N,\, t\geq 0.\end{aligned}\end{equation}

Here $B(k)$, $k\in S^N$, are independent Brownian motions and the matrix
 \begin{align}\label{E1.02}
\mathcal A^N(k,l)=\begin{cases}
	\frac 1 N,&\mfalls k\neq l,\\
	\frac{1}{N}-1,&\mfalls k=l,
	\end{cases}
\end{align}
is the transition operator for migration between sites. The function $g:I\to[0,\infty)$ is the so-called diffusion coefficient and is assumed to be sufficiently smooth and well-behaved. We will denote the continuous time transition matrix of $\CA^N$ by
\begin{equation}\label{E1.03}
\etAN(k,l)=\frac1N\big(1-e^{-t}\big)+\1_{\{k=l\}}\,e^{-t}.
\end{equation}
Note that $\etAN$ is the time $t$ transition matrix of a continuous time Markov chain on $S^N$ that makes uniformly distributed jumps at rate $1$ and this is the chain defined by the $q$-matrix $\CA^N$.
\clearpage
Most prominent examples are
\begin{itemize}
	\item[(a)] $I=\R$, $g_\gamma(x)\equiv \gamma>0$, so-called ``interacting Brownian motions'',
	\item[(b)] $I=[0,\infty)$, $g_\gamma(x)=\gamma\, x^2$, so-called ``parabolic Anderson model with Brownian potential'',
	\item[(c)] $I=[0,1]$, $g_\gamma(x)=\gamma\, x\,(1-x)$, so-called ``stepping stone model''.
\end{itemize}

Let
\begin{equation}
\label{E1.04}
\Theta^N_t:=\frac{1}{N}\sum_{k\in S^N}X^N_t(k),\qquad t\geq0,\end{equation}

be the average process of the system \eqref{E1.01}. Due to the choice of $\mathcal A^N$ the matrix multiplication in \eqref{E1.01} can be rewritten as
\begin{equation}
\label{E1.05}
	 dX^N_t(k)=\big(\Theta^N_t-X^N_t(k)\big)\,dt+\sqrt{g(X^N_t(k))}\,dB_t(k),\qquad k\in S^N,\, t\geq 0.
\end{equation}
We give a very rough sketch of the basic idea of the finite systems scheme. Assume that $\Theta^N_0$ converges weakly to some value $\theta$ as $N\to\infty$.  By a law of large numbers, we get $\lim_{N\to\infty}\Theta^N_t=\theta$ for all $t\geq0$ and hence, formally, the equation \eqref{E1.05} for one coordinate converges to
\begin{equation}
\label{E1.06}
	 dX_t(k)=\big(\theta-X_t(k)\big)\,dt+\sqrt{g(X_t(k))}\,dB_t(k),\qquad k\in\N, \,t\geq 0,
\end{equation}
as $N\to\infty$. The diffusions $X(k)$, $k\in\N$, are now independent and (under suitable assumptions on $g$) converge for $t\to\infty$ to an ergodic equilibrium distribution $\nu_\theta=\nu_\theta^g$.\\ \smallskip

Now an appropriate time-rescaling gives a non-trivial limit for $\Theta^N$. More precisely, for $\beta^N:=N$, under mild assumptions on $g$, the time-rescaled process $(\Theta^N_{\beta^Nt})_{t\geq0}$ converges to a diffusion process $\Theta$ which is the solution of the stochastic differential equation
\begin{equation}\label{E1.07}
	d\Theta_t=\sqrt{g^*(\Theta_t)}\,dB_t,\qquad  \,t\geq 0.
\end{equation}
Here, $B$ is a Brownian motion and
\begin{equation}
\label{E1.08}
g^*(\theta)=\int g(x)\,\nu^g_\theta(dx)
\end{equation}
is the (approximate and up to a factor $1/N^2$) mean contribution of a single coordinate $X^N(k)$ to the square variation process $\langle \Theta^N\rangle$.

The nonlinear map $g\mapsto g^*$ was studied in a series of papers by \cite{BaillonClementGrevenHollander1993,BaillonClementGrevenHollander1995} and (in a multi-dimensional situation) \cite{DawsonGrevenHollanderSunSwart2008}. In particular, the fixed shapes (i.e. $g^*=c\cdot g$ for some $c>0$) are (uniquely up to linear factors) identified as
\begin{itemize}
\item $g(x)=1$ if $I=\R$,
\item $g(x)=x$ if $I=[0,\infty)$,
\item $g(x)=x(1-x)$ if $I=[0,1]$.
\end{itemize}
In the situation of two-dimensional interacting models, formally corresponding to \eqref{E1.01} with $I=\R_+^2$, the only non-trivial fixed shape is $g((u,v))=u\cdot v$ for $u,v\geq0$. For this situation, the finite systems scheme was developed by \cite{CoxDawsonGreven2004}.\\

We see that the average process $\Theta^N$ fluctuates on a slower time scale than the individual coordinate processes $X^N(k)$. Hence, from time $\beta^Nt$ to $\beta^Nt+s$ (with $s>0$ large) the coordinates have enough time to converge (independently) to their equilibrium state $\nu^g_{\theta'}$ (given $\Theta^N_{\beta^Nt}=\theta'$). Thus, we should have (in the sense of weak convergence of finite dimensional distributions)
\begin{equation}
\label{E1.09}
\mathcal{L}\Big(\big(\Theta^N_{\beta^Nt},(X^N_{\beta^Nt}(k))_{k\in S^N}\big)\Big)\limN\int P^g_t(\theta,d\theta')\left(\delta_{\theta'}\otimes (\nu^g_{\theta'})^{\otimes\N}\right),
\end{equation}
where $P^g_t(x,dy)$ denotes the transition probabilities of $\Theta$ from \eqref{E1.07}.
One could even expect that the full processes $X^N$ (and not only the marginal at time $\beta^Nt$) converge. To be more precise, denote by $\check \nu^g_\theta$ the distribution of the process $(X_t)_{t\geq0}$, where $X$ is the stationary solution of \eqref{E1.06}. Then, under some mild regularity conditions on $g$,
\begin{equation}
\label{E1.10}
\mathcal{L}\Big(\big(\Theta^N_{\beta^Nt},((X^N_{\beta^Nt+s}(k))_{s\geq0})_{k\in S^N}\big)\Big)\limN\int P^g_t(\theta,d\theta')\left(\delta_{\theta'}\otimes (\check\nu^g_{\theta'})^{\otimes\N}\right).
\end{equation}
The statements \eqref{E1.09} and \eqref{E1.10} are often referred to as (mean field) finite systems scheme. The formal statements are proved (in greater generality) in \cite[Theorem 1]{DawsonGreven1993b} and for a two-dimensional setting in \cite{CoxDawsonGreven2004}.

\subsection{The infinite rate renormalization}
\label{S1.2}
Consider first the case of one-di\-men\-sion\-al interacting diffusions with compact $I=[0,1]$. For the renormalization map $g\mapsto g^*$, the only fixed shape is $g(x)=x(1-x)$, that is, the Wright-Fisher diffusion. However, the Wright-Fisher diffusion also pops up as the result of a renormalization procedure that we explain now. Consider the solution $X^{N,\gamma}$ of \eqref{E1.01} with $g$ replaced by $g^\gamma=\gamma\cdot g$ for some $\gamma>0$. We assume that $g(x)=0$ for $x=0,1$ and $g(x)>0$ for $x\in(0,1)$. One can show that, as $\gamma\to\infty$, $X^{N,\gamma}$ converges (for example in finite dimensional distributions or in the Meyer-Zheng pseudo-path topology) to a process $X^N$ with values in $\{0,1\}^{S^N}$. In fact, in the interior $(0,1)$ of $I$, the coordinate processes fluctuate faster and faster and are thus (in the limit) driven to the boundary of $I$ immediately.
 Furthermore, since
$$X^{N,\gamma}_t(k)-X^{N,\gamma}_0(k)-\int_0^t\CA^NX^{N,\gamma}_s(k)\,ds,\qquad t\geq0,\,k\in S^N,$$
is a martingale, it can be seen that also
$$X^{N}_t(k)-X^{N}_0(k)-\int_0^t\CA^NX^{N}_s(k)\,ds,\qquad t\geq0,\,k\in S^N,$$
is a martingale. From the martingale property it can be deduced that $X^{N}$ is a voter model with a symmetric updating mechanism. With this convergence in mind, the voter process can be seen as an "infinite rate" ($\gamma=\infty$) model.
The average process $\Theta^{X^N}$ of the voter process is known as the Moran model from population genetics. It is well known that $(\Theta^{X^N}_{Nt})_{t\geq0}$ converges in finite dimensional distributions (and even in the Skorohod topology) to the Wright-Fisher diffusion, that is, to the solution of the stochastic differential equation
$$dY_t=\sqrt{Y_t(1-Y_t)}\,dB_t.$$
Here we see that the diffusion function $g(x)=x(1-x)$ shows up in the limiting equation for the infinite rate renormalization scheme if $I=[0,1]$. One could try to find also the fixed shapes for $I=[0,\infty)$ and $I=\R$ as limits of an infinite rate renormalization. However, a little thought shows that the limit as $\gamma\to\infty$ is either trivial ($I=[0,\infty)$) or not well defined ($I=\R$). Hence, for interacting diffusions which are one-dimensional at each site not much more can be done.\\ \smallskip

The situation becomes more interesting in the two-dimensional setting corresponding formally to $I=\R_+^2$. Similarly to the universal convergence to the voter process described above, in the two-dimensional setting, under some conditions on $g$,
 there is a non-trivial discontinuous limiting process $X^{N}$ if for $g^\gamma=\gamma g$ we let $\gamma\to\infty$. Similarly to the voter process which takes values at each site in the boundary $\{0,1\}$ of $[0,1]$, the universal limiting process $X^{N}$ takes values in the boundary of $\R_+^2$, that is
 \begin{align*}
E:=[0,\infty)^2\setminus (0,\infty)^2.
\end{align*}
\begin{rem}
	For $x=(x_1,x_2)\in \R^2_+$ we call the two coordinates the types. If $x\in E$ with $x_2=0$ we say $x$ is of type 1, if $x_1=0$ we say $x$ is of type 2.
\end{rem}
The limiting process $X^{N}$ does not depend on the details of the diffusion function $g$ as long as $g$ is strictly positive in $(0,\infty)^2$ and $0$ at the quadrant's boundary  $E$ (and is sufficiently regular to allow existence of
 a solution to  SDEs). See \cite[Theorem 1.5]{KM2} for a formal statement. The process $X^{N}$ is called \emph{infinite rate mutually catalytic branching process} or MCB($\infty$) since it was introduced as infinite branching rate limit of mutually catalytic branching processes as will be discussed  in the next subsection. \\ \smallskip

We will show that there is a time scale $\beta^N$ such that $(X^{N}_{\beta^Nt})_{t\geq0}$ converges in the Skorohod topology to a process that solves the two-dimensional analogue of \eqref{E1.01} with $g((u,v))=(8/\pi) uv$, the fixed shape of the transformation $g\mapsto g^*$ in two dimensions. Furthermore, we will develop the finite systems scheme in the sense of \eqref{E1.09} and \eqref{E1.10}. Unlike the voter model, the limiting process $X^{N}$ lacks second moments (but possesses all $p$th moments for $p<2$) and is described by a jump type stochastic differential equation. Hence, usual standard arguments of computing the square variation process do not work. Furthermore, the typical scaling in the presence of variances does not work properly and we have to employ a logarithmic correction: \begin{equation}
\label{E1.12}\beta^N=\frac{N}{\log N}.
\end{equation}

\subsection{Mutually catalytic branching processes}\label{S1.3}
In this subsection we define the universal infinite rate limiting process $X^{N}$ of two-dimensional interacting diffusion processes on $\R_+^2$ with sites space $S^N$. The process is introduced as infinite rate limit of mutually catalytic branching processes and can be characterized as solution to a stochastic equation.\\\smallskip

Dawson and Perkins \cite{DawsonPerkins1998} introduced a spatial two-type branching model where the local branching rate of type 1 is proportional to the amount of type 2 particles at the same site and vice versa. Furthermore, the infinitesimal individuals migrate through space according to some Markov kernel. In our setting with mean-field interaction $\mathcal A^N$ on $S^N$, the model can be described as the (unique weak) solution of the system of stochastic differential equations driven by independent Brownian motions
\begin{align}
\label{E1.13}
		d X^{N,\gamma,i}_t(k)=\mathcal A^N X^{N,\gamma,i}_t(k)\,dt+\sqrt{\gamma \, X^{N,\gamma,1}(k)\, X^{N,\gamma,2}(k)}\;dB^i_t(k)
	\end{align}
for $i=1,2$, $k\in S^N$, $\gamma>0$ and $t\geq 0$. This model is called \emph{mutually catalytic branching model with finite rate} $\gamma$, or MCB($\gamma$), and solutions
 $$X^{N,\gamma}_t(k)=\big(X^{N,\gamma,1}_t(k),X^{N,\gamma,2}_t(k)\big)\in\R_+^2$$
are called  \emph{mutually catalytic branching processes}. As one can see, this is a
particular case of a two-dimensional interacting diffusion model with $g(u,v)=\gamma uv$. \\

Now we give the description of the infinite rate mutually catalytic branching process MCB($\infty$). If in \eqref{E1.13} we let $\gamma\to\infty$, then, heuristically, the single coordinates $X^{N,\gamma}_t(k)$ are driven to the boundary $E$ of $\R_+^2$ immediately. Since the diffusion is isotropic, the distribution of the exit point does not depend on the specific diffusion coefficient and  thus is the same as for planar Brownian motion  $W=(W^1,W^2)$ on $[0,\infty)^2$, started at $W_0=x$ (this is a consequence of the Dubins-Schwarz theorem). Let $Q_{x}(dy)$ denote the harmonic measure of planar Brownian motion on $\R_+^2$, started at $x\in\R_+^2$. That is, $Q_{x}(dy)$ is the distribution of the exit point of a planar Brownian motion in the quadrant started at $x$. Loosely speaking, if site $k$ is populated by type 2, then migration of type 2 individuals results in deterministic (discrete space heat flow) changes while type 1 immigration results in jump activity.

Using the explicit Lebesgue densities of the harmonic measures $Q_{x}$ for $x\in(0,\infty)^2$ (see, e.g., \cite[Lemma 1.2]{KM2}), it is easy to show that for $x=(x_1,0)\in E$, the vague limit  $$\NU_x:=\lim_{\varepsilon\to0}\frac1\varepsilon Q_{(x_1,\varepsilon)}$$
exists  on $E\setminus\{x\}$. The analogous statement holds for $x=(0,x_2)\in E$. The measure $\NU_x$ can be thought of as the prototypic measure for jumps away from $x$ when there is an immigration of the respective other type. Due to symmetry and a scaling relation, all the measures $\NU_x$ are simple transformations (described below implicitly, see also \cite{KlenkeMytnik2010}, discussion before (5.5)) of the measure $\NU:=\NU_{(1,0)}$. This measure $\NU$ on $E$ can be explicitly described in terms of its Lebesgue densities

\begin{align}\label{E1.14}
			\NU(dy)=\begin{cases}
				 \displaystyle\frac{4}{\pi}\frac{y_1}{(1-y_1)^2(1+y_1)^2}\,dy_1,&\mfalls y_1\geq0,\,y_2=0,\\[4mm]
				 \displaystyle\frac{4}{\pi}\frac{y_2}{(1+y_2^2)^2}\,dy_2,&\mfalls y_1=0,\, y_2\geq0,
			\end{cases}
\end{align}	
on $E$. Properties of $\nu$ are collected in some lemmas in the appendix. The jump structure of the MCB($\infty$) process $X^N$ is described by means of a Poisson point process $\mathcal N$ on $\N\times E\times \R^+\times \R^+$ with intensity measure
		\begin{align*}
			\mathcal N^{\prime}= \ell\otimes \NU\otimes\lambda\otimes\lambda.
		\end{align*}
 Here, $\ell$ denotes the counting measure on $\N$ and $\lambda$ the Lebesgue measure on $\R^+$. In order to describe the intensity of jumps depending on the current state of the system, let
$$
\begin{aligned}
I^N_t(k)&:=I^{N,1}_t(k)+I^{N,2}_t(k)\\&:=\1_{\{X^{N,2}_t(k)>0\}}\frac{\mathcal A^NX^{N,1}_t(k)}{X^{N,2}_t(k)}+\1_{\{X^{N,1}_t(k)>0\}}\frac{\mathcal A^NX^{N,2}_t(k)}{X^{N,1}_t(k)},\quad k\in S^N.		
\end{aligned}
$$
Note that $I^N_t(k)$ is well-defined because either $X^N_t(k)$ is of type $1$, i.e. $X^{N,1}_t(k)> 0$ and $X^{N,2}_t(k)=0$, or of type $2$, i.e. $X^{N,2}_t(k)> 0$ and $X^{N,1}_t(k)=0$. Since the off-diagonal entries of $\CA^N$ are nonnegative, the rates $I^{N,1}_t(k)$, $I^{N,2}_t(k)$ and $I^N_t(k)$ are nonnegative.\\

The jumps of $\mathrm{MCB}(\infty)$ are governed by the function $J:E\times E\to \R^2$
		\begin{align*}
			J\big( y, x\big)=\left(J_1( y, x)\atop J_2( y, x) \right),
		\end{align*}
		where the coordinate jumps
		\begin{align*}
			J_1( y,  x)&=\begin{cases}
					(y_1-1)x_1,&\mfalls  x=\left(x_1\atop 0\right),\\[2mm]
					y_2 x_2,&\mfalls  x=\left(0\atop x_2\right),
				\end{cases}
\end{align*}
and
\phantom{changed J(x,y) to J(y,x)}
\begin{align*}
			J_2( y,  x)&=\begin{cases}
					y_2x_1,&\mfalls  x=\left(x_1\atop 0\right),\\[2mm]
					(y_1-1) x_2,&\mfalls  x=\left(0\atop x_2\right),
				\end{cases}
		\end{align*}
		depend on the state $ x=\left(x_1\atop x_2\right)$ of the system and a point $ y=\left(y_1 \atop y_2\right )$ is chosen from $E$ according to $\NU$. \smallskip
		
		  The system of stochastic equations characterizing $\mathrm{MCB}(\infty)$ on $S^N$ is
		\begin{equation}\label{E1.15}
\begin{aligned}			 X^{N,i}_t(k)&=X^{N,i}_0(k)+\int_0^t\mathcal{A}^NX^{N,i}_s(k)\,ds\\
&\quad+\int_0^t\int_0^{I^N_{s-}(k)}\int_{E}J_i( y,X^N_{s-}(k))\big(\mathcal{N}-\mathcal{N}^{\prime}\big)\big(\{k\},dy,dr,ds\big),
		\end{aligned}\end{equation}
		for $i=1,2$, $k\in S^N$ and $t\geq 0$. The idea is that each coordinate $X^N_t(k)$ experiences a drift towards the mean of all coordinates. In addition, it is a (non-trivial) consequence of the particular form of the jump function $J$ that solutions are forced to remain only at the boundary $E$ of $[0,\infty)^2$: Jumps go from $E$ to $E$ and the compensator cancels with the drift (compare Section~2 of \cite{KM2}). Also note that the $dr$-contribution does not play a role for the jump target but instead only determines the jump \emph{rate} which is proportional to $I^N$.\smallskip

\begin{rem}\label{rr}
To facilitate the understanding of Equation \eqref{E1.15} let us recall the interpretation of the jump mechanism through generalized voter processes instead of types (see \cite{DoeringMytnik2012}): If just before time $t$ voter $k$ has opinion A with a strength of conviction $x_1$, that is $X^N_{t-}(k)=\left (x_1\atop 0\right)$, then with a rate which is the total conviction strength of all neighbors of opposite opinion B relativized by the conviction strength $x_1$ of voter $k$, voter $k$ chooses to reconsider his/her opinion: he/she chooses a new opinion according to $\NU$. If the point $y\in E$ chosen by $\NU$ takes the form $(y_1,0)$ then the opinion of voter $k$ does not change but the strength is multiplied by $y_1$. Conversely, if the chosen point is $(0,y_2)$ then the opinion of voter $k$ changes and the strength is multiplied by $y_2$.  Hence, there are four possible types of jumps:
\begin{itemize}
	\item opinion A $\rightarrow$ opinion A:\\ $\,\,\,\left(x \atop 0\right)\mapsto \left(y_1 x\atop 0\right)=\left(x \atop 0\right)+\left(y_1-1\atop 0\right)x=\left( x\atop 0\right)+J\left(\left(y_1\atop 0\right),\left(x\atop 0\right)\right)$
	\item opinion A $\rightarrow$ opinion B:\\ $\,\,\,\left(x\atop 0\right)\mapsto \left(0\atop y_2x\right)=\left(x\atop 0\right)+\left(-1\atop y_2\right)x=\left(x\atop 0\right)+J\left(\left(0\atop y_2\right),\left(x\atop 0\right)\right)$
	\item opinion B $\rightarrow$ opinion B:\\ $\,\,\,\left(0\atop x\right)\mapsto \left(0 \atop y_1 x\right)=\left(0 \atop x\right)+\left(0\atop y_1-1\right)x=\left(0\atop x\right)+J\left(\left(y_1\atop 0\right),\left(0\atop x\right)\right)$
	\item opinion B $\rightarrow$ opinion A:\\ $\,\,\,\left(0\atop x\right)\mapsto \left(y_2 x\atop 0\right)=\left(0\atop x\right)+\left(y_2\atop -1\right)x=\left(0\atop x\right)+J\left(\left(0\atop y_2\right),\left( 0\atop x\right)\right)$
\end{itemize}
By definition, $\NU$ has infinite mass on the positive part of the $x$-axis with a pole at $\left(1 \atop 0\right)$ whereas the mass of $\NU$ on the positive $y$-axis is finite. This means that in finite time there are infinitely many tiny changes in strength of conviction without changing the opinion whereas there are only finitely many changes of opinion.
\end{rem}
		
	Here is the main theorem for convergence of $\mathrm{MCB}(\gamma)$ to $\mathrm{MCB}(\infty)$:
\setcounter{thm}{-1}
		\begin{thm}[{\cite[Theorem 1.5]{KM2}}]\label{111}
			If $X^N_0\in E^{S^N}$, then \eqref{E1.15} has a unique weak solution $(X^N_t)_{t\geq 0}$ with $X^N_t(k)\in E$ for all $t>0$ and $k\in S^N$. The unique solution of \eqref{E1.15} is called $\mathrm{MCB}(\infty)$ and
			\begin{align*}
				\mathrm{MCB}(\gamma) \stackrel{\gamma\to\infty}{\Longrightarrow} \mathrm{MCB}(\infty)
			\end{align*}
			in the sense of weak convergence in the Meyer-Zheng topology.
		\end{thm}
				The claimed universality of MCB($\infty$) was also established in Theorem 1.5 of \cite{KM2}: In fact, the diffusion function is not necessarily $g(u,v)=uv$ as for mutually catalytic branching. The diffusion function only needs to vanish on $E$ and be positive on $(0,\infty)^2$.		
\subsection{Our Results}
\label{S1.4}
We will establish a finite systems scheme for MCB($\infty$) in the sense of \eqref{E1.09} and \eqref{E1.10} with $\beta^N=N/\log N$ as in~\eqref{E1.12}
and with $P^g$ replaced by the semigroup $P^{8/\pi}$ of MCB($8/\pi$). Since the major part of the work is proving convergence of the average process, we formulate this statement in a separate theorem first.

Let
$$Z^{N,i}_t:=\frac{1}{N}\sumkS   X^{N,i}_t(k),\qquad i=1,2,\,k\in S^N,\,t\geq0,$$
and $\bZ^N_t:=(Z^{N,1}_t,Z^{N,2}_t)$
be the average processes.
Define the time-rescaled processes
\begin{equation}
\label{E1.16}
\rX^{N,i}_t(k):=X^{N,i}_{\beta^Nt}(k),\qquad k\in S^N,\,i=1,2,\,t\geq0,
\end{equation}
and $\brZ^N_t:=\bZ^N_{\beta^Nt},$ where $\beta^N$ is given by~\eqref{E1.12}.
We will also write
$\rI^{N}_t(k):=I^N_{\beta^Nt}(k)$ for the scaled jump rates.

\begin{thm}\label{T1}
	Suppose $X^{N}$ is $\mathrm{MCB}(\infty)$ on $S^N$ with mean-field interaction $\mathcal A^N$.
	Assume that $\sup_{N\in \N} \E[(Z^{N,i}_0)^2]<\infty$ for $i=1,2$ and $\brZ^{N}_0\Rightarrow z\in \R_+^2$ as $N\to\infty$. Furthermore, assume that there exists a $p\in(1,2)$ such that
\begin{equation}\label{E1.18}
C_p:=1+\sup_{N\in\N,\,i=1,2}\frac1N\sumkS  \E\big[(X^{N,i}_0(k))^p\big]\,<\infty.\end{equation}
Then
	\begin{align*}
		\big(\brZ^N_t\big)_{t\geq 0}\wlimN(\bZ_t)_{t\geq 0},
	\end{align*}
	weakly in the Skorokhod space on $\R_+^2$. Here, $\bZ=(Z^1 , Z^2)$ takes values in $\R_+^2$ and is the unique strong solution of
	\begin{align}\label{E1.19}
	dZ^i_t=\sqrt{\gamma^*\, Z^1_t \, Z^2_t}\,dB_t^i,\qquad i=1,2,\,t\geq0,
	\end{align}
	driven by independent Brownian motions $B^1, B^2$ with initial condition $\bZ_0=z$. The branching rate is $\gamma^*=\frac{8}{\pi}$.
\end{thm}

Now we come to the formulation of the finite systems scheme. Denote by $(P^{8/\pi}_t)_{t\geq0}$ the semigroup of MCB($8/\pi$), that is,
$$P^{8/\pi}_t(z,dz')=\P[\bZ_t\in dz'\,|\, \bZ_0=z]$$
with $\bZ$ from Theorem~\ref{T1}. As an analogue to \eqref{E1.06}, we consider the (unique strong) solution in $[0,\infty)^2$ of the equation
\begin{equation}
\label{E1.20}
dY^{\theta,\gamma,i}_t=\big(\theta^i-Y^{\theta,\gamma,i}_t\big)dt+\sqrt{\gamma Y^{\theta,\gamma,1}_tY^{\theta,\gamma,2}_t}\,dW^i_t,\qquad i=1,2,\,t\geq0,
\end{equation}
where $W^1$, $W^2$ are two independent Brownian motions and $\theta=(\theta^1,\theta^2)\in[0,\infty)^2$.
It was shown in \cite[Theorem 1.3]{KlenkeMytnik2010} that the limit $Y^\theta$ of $Y^{\theta,\gamma}=(Y^{\theta,\gamma,1},Y^{\theta,\gamma,2})$ as $\gamma\to\infty$ exists in the sense of weak convergence of finite dimensional distributions. Furthermore, the transition semigroup can be computed explicitly
\begin{equation}
\label{E1.21}
\P[Y^\theta_t\in dy'|Y^\theta_0=y]=Q_{e^{-t}y+(1-e^{-t})\theta}(dy').
\end{equation}
The invariant distribution of $Y^\theta$ is the harmonic measure $Q_\theta$ of planar Brownian motion on $(\R_+)^2$ started at $\theta$. Furthermore, let $\check Q_\theta$ denote the law of the stationary process, that is of $Y^\theta$ started with initial distribution $Q_\theta$. Recall that $\beta^N=N/\log N$.
\begin{thm}
\label{T2}
Under the assumptions of Theorem~\ref{T1}, we have,
in the sense of weak convergence of the finite dimensional distributions: for any $t>0$,
(i)
\begin{equation}
\label{E1.22}
\CL\big((\brZ^N_t,(\rX^N_t(k):k\in S^N))\big)\limN \int P^{8/\pi}_t(z,dz')\left(\delta_{z'}\otimes Q_{z'}^{\otimes\N}\right)
\end{equation}
(ii) and
\begin{equation}
\label{E1.23}
\CL\big((\bZ^N_{\beta^Nt},((X^N_{\beta^Nt+s}(k))_{s\geq0}:k\in S^N))\big)\limN \int P^{8/\pi}_t(z,dz')\left(\delta_{z'}\otimes \check Q_{z'}^{\otimes\N}\right).
\end{equation}
\end{thm}
Theorem~\ref{T2} shows that in fact a mean field finite systems scheme in the of \eqref{E1.09} and \eqref{E1.10} holds.

\subsection{Outline}
\label{S1.5}
The proof of Theorem~\ref{T1} follows a general strategy:
\begin{itemize}
\item[(i)]
Prove tightness of the sequence $(\brZ^N)_{N\in \N}$;
\item[(ii)]Prove that any limit point of $(\brZ^N)_{N\in \N}$ is a weak solution of the SDE \eqref{E1.19} with $\gamma^*=\frac{8}{\pi}$;
\item[(iii)] Prove that all limit points are equal.
\end{itemize}
Step (i) is carried out by fine moment estimates, (ii) is proved using the method of characteristics for semimartingales and  (iii) is a consequence of (ii) and the strong uniqueness of solutions to the SDE \eqref{E1.19}.\smallskip

The proof of Theorem~\ref{T2} makes use of an approximate duality of MCB($\infty$) to some deterministic process in order to compare the coordinate processes of $X^N$ with $Y^\theta$ from \eqref{E1.21}.\medskip

The article is organized as follows: In Section~\ref{S2}, the proof of Theorem \ref{T1} is given: We start with a rough heuristics in Section~\ref{S2.1}. Auxiliary moment estimates are gathered in Section~\ref{S2.3}, tightness arguments are given in Section~\ref{S2.4} and the final convergence proof is given in Section~\ref{S2.5}. In order to make the article more accessible to the reader not familiar with the general theory of semimartingales, definitions of semimartingale characteristics are recalled in Section~\ref{S2.2}. Finally, in Section~\ref{S3}, we establish the approximate duality to some deterministic process and prove Theorem~\ref{T2}.\\

\textbf{Notation.} Throughout this article, $C$ denotes a generic constant that can vary from line to line.


\section{Proof of Theorem~\ref{T1}}

\label{S2}
\subsection{Heuristics}
\label{S2.1}
In this subsection we give a rough and oversimplified idea why $\beta^N=N/\log(N)$ is the right scaling and why the limiting process is the mutually catalytic branching process with rate $8/\pi$.
Assume that $N$ is even and that
$$\rX^N_{t}(k)=\begin{cases}
\big(2\rZ^{N,1}_t,0\big)&\mfalls k\leq N/2,\\[2mm]
\big(0,2\rZ^{N,2}_t\big)&\mfalls k> N/2.
\end{cases}
$$
That is, at time $\beta^N t$, the MCB($\infty$) process is such that
\begin{itemize}
\item
type 1 is constant on sites $k\leq N/2$,
\item
type 2 is constant on sites $k> N/2$.
\end{itemize}
 We next argue that large jumps disappear in the limit whereas small jumps lead to a quadratic variation part including our factor $8/\pi$. As explained in Remark \ref{rr} there are different sorts of jumps: big, small, no change of types and change of types. Analyzing their effects separately explains the limiting process. By symmetry, it is enough to consider the changes in the first coordinate $\rZ^{N,1}_t$.\\

\textbf{Large Jumps; case 1.} Jumps of $\rZ^{N,1}_t$ of size larger than $\varepsilon$ due to a jump at some coordinate $k\leq N/2$ (no change of type).

The jumps of $\rZ^{N,1}$ of size $\varepsilon$ at time $t$ are due to all jumps of $X^N$ of size $\varepsilon N$ at all sites $k\leq N/2$. We calculate the rate of such jumps: plugging the definition of $\mathcal A^N$ into the definition of $I^N$, multiplying with $N/\log(N)$ for the time-scaling, multiplying by $N/2$ since there are $N/2$ possibilities to have such a jump and, finally, multiplying by the mass of the intensity measure on $E$ so that such a jump occurs gives the total jump rate
$$\begin{aligned}
 \frac{N}{\log N}&\;\frac{N}{2}\;\frac{\rZ^{N,2}_t}{2\rZ^{N,1}_t}\;
\nu\big(\big\{y:\,|y_1-1|>\ve N/(2\rZ^{N,1}_t)\big\}\big)\\
&\leq \frac{N}{\log N}\;\frac{N}{2}\;\frac{\rZ^{N,2}_t}{2\rZ^{N,1}_t}\;\frac2\pi(2\rZ^{N,1}_t/(\ve N))^2\\
&= \frac{1}{\log N}\,\frac{1}{\ve^2}\,\frac2\pi \,\rZ^{N,1}_t\rZ^{N,2}_t\limN0,
\end{aligned}
$$
where the inequality for $\nu$ comes from the appendix and for the last convergence to zero we used the stochastic boundedness of  the sequences $\rZ^{N,i}_t\,, i=1,2$. The stochastic boundedness is a consequence of tightness in $N$ which is proved before using moment bounds.

\textbf{Large Jumps; case 2.} Jumps of $\rZ^{N,1}$ of size larger than $\varepsilon$ at time $t$ due to a jump at some coordinate $k> N/2$ (implying a change of type from 2 to 1).

The jump rate for such jumps is calculated and estimated as above:
$$\begin{aligned}
\frac{N}{\log N}&\;\rZ^{N,1}_t\;\frac{N}{2}\;\frac{1}{2\rZ^{N,2}_t}\;
\nu\big(\big\{y:\,y_2>\ve N/(2\rZ^{N,2}_t)\big\}\big)\\
&\leq \frac{N}{\log N}\;\rZ^{N,2}_t\;\frac{N}{2}\;\frac{1}{2\rZ^{N,2}_t}\;\frac2\pi\big(2\rZ^{N,2}_t/(\ve N)\big)^2\\
&= \frac{1}{\log N}\,\frac{1}{\ve^2}\,\frac2\pi \,\rZ^{N,1}_t\rZ^{N,2}_t\limN0.
\end{aligned}
$$

\textbf{Quadratic Variation; case 1.} The quadratic variation of $\rZ^{N,1}$ at time $t$ due to jumps of size $\leq \ve$ originating from jumps of $X^N$ of size $\leq \ve N$ at sites $k\leq N/2$ grows at rate (compare Lemma~\ref{LA.4} for the asymptotic equivalence)
$$
\begin{aligned}
\frac{N}{\log N}&\;\frac{N}{2}\;\frac{\rZ^{N,2}_t}{2\rZ^{N,1}_t}\;\int\limits_{\{|y_1-1|<\ve N/(2\rZ^{N,1}_t)\}}(y_1-1)^2\left(\frac{2\rZ^{N,1}_t}{N}\right)^2\,\nu(dy)\\
&=\frac{1}{\log N}\rZ^{N,1}_t\rZ^{N,2}_t\int\limits_{\{|y_1-1|<\ve N/(2\rZ^{N,1}_t)\}}(y_1-1)^2\,\nu(dy)\\
&\stackrel{N\to\infty}{\sim}\frac{1}{\log N}\rZ^{N,1}_t\rZ^{N,2}_t\frac4\pi\log\big(\ve N/(2\rZ^{N,1}_t)\big)\\
&\limN\frac4\pi \rZ^{N,1}_t\rZ^{N,2}_t.
\end{aligned}$$
\textbf{Quadratic Variation; case 2.} The quadratic variation of $\rZ^{N,1}$ at time $t$ due to jumps of size $\leq \ve$ originating from jumps of $X^N$ of size $\leq \ve N$ at sites $k> N/2$  grows at rate (compare Lemma~\ref{LA.3} for the asymptotic equivalence)
$$
\begin{aligned}
\frac{N}{\log N}&\;\frac{N}{2}\;\frac{\rZ^{N,1}_t}{2\rZ^{N,2}_t}\;\int\limits_{\{y_2<\ve N/(2\rZ^{N,2}_t)\}}y_2^2\left(\frac{2\rZ^{N,2}_t}{N}\right)^2\,\nu(dy)\\
&=\frac{1}{\log N}\rZ^{N,1}_t\rZ^{N,2}_t\int\limits_{\{y_2<\ve N/(2\rZ^{N,2}_t)\}}y_2^2\,\nu(dy)\\
&\stackrel{N\to\infty}{\sim}\frac{1}{\log N}\rZ^{N,1}_t\rZ^{N,2}_t\frac4\pi\log\big(\ve N/(2\rZ^{N,2}_t)\big)\\
&\limN\frac4\pi \rZ^{N,1}_t\rZ^{N,2}_t.
\end{aligned}$$

Summing up, we see that asymptotically there are no jumps of size $\geq\ve$ and the square variation grows at rate $\frac8\pi \rZ^{N,1}_t\rZ^{N,2}_t$. However, this is exactly the characterization of MCB($8/\pi$). In the next subsections we make this reasoning precise and rigorous by applying general semimartingale theory.

\subsection{Semimartingale Setup}
\label{S2.2}
The proof of Theorem~\ref{T1} is based on general limit theorems for semimartingales that can be found in \cite{JacodShiryaev2003}. For convenience, we collect here some basic facts and definitions. All processes are defined on the state space
	\begin{align*}
		\mathbb{D}=\Big\{ \alpha:\R_+\to \R_{+}\times \R_{+}\,,\, \alpha \text{ right-continuous with left limits}\Big\}
	\end{align*}
	equipped with the Skorohod topology. For definitions and properties, the reader might consult Chapter~VI of \cite{JacodShiryaev2003}. Since all appearing processes are semimartingales, we can use criteria for convergence based on the semimartingale characteristic triplet. In order to describe the triplet of a two-dimensional semimartingale, let $h=\left(h_1\atop h_2\right):\R^2\to \R^2$ be a truncation function (that is, compactly supported with $h(x)=x$ around the origin). The truncation function is fixed in the background and all results of interest are independent of its choice. The characteristic triplet (see \cite[Definition II.2.6]{JacodShiryaev2003}) of a two-dimensional semimartingale $\mathbf Y=\left(Y^1\atop Y^2\right)$ on a filtered probability space $(\Omega, \mathcal F, (\mathcal F_t)_{t\geq 0},\P)$ is the triplet $(\mathbf B,\mathbf C,\MU)$ consisting of
\begin{itemize}
	\item $\mathbf B=\left(B^1\atop B^2\right)$, a predictable process of bounded variation,
	\item $\mathbf C=(C_{i,j})_{i,j=1,2}=(\langle Y^{i,c}, Y^{j,c}\rangle)_{i,j=1,2}$, where $\mathbf Y^{c}=\left(Y^{1,c}\atop Y^{2,c}\right)$ is the continuous martingale part of $\mathbf Y$,
	\item the compensator measure $\MU$ on $(\R^2\times \R_+, \mathcal B(\R^2\times \R_+))$ of the point process $\MU^\Delta$ of jumps of $\mathbf Y$, also abbreviated as $\MU_t(\,\ARG\,)=\MU( \,\ARG\,\times[0,t])$,
\end{itemize}
so that the canonical representation of $\mathbf Y$ holds:
\begin{align*}
	\mathbf Y_t=\mathbf Y_0+\mathbf Y^c_t+h\ast (\MU^\Delta_t-\MU_t)+\bar h\ast \MU^\Delta_t+\mathbf B_t,\qquad t\geq 0,
\end{align*}
where $\bar h(x)=x-h(x)$ and $\ast$ denotes integration against point processes (see for instance Section~II.1 of \cite{JacodShiryaev2003}). Note that - comparing with the special case of a L\'evy process written in L\'evy-It\^o form - the canonical representation looks more familiar when $h(x)=x\1_{\{|x|\leq 1\}}$, namely,

\begin{equation}\label{E2.01}\begin{aligned}
	\mathbf Y_t&=\mathbf Y_0+\mathbf Y^c_t+\int_0^t \int_{|x|\leq 1} x\, (\MU^\Delta-\MU)(dx,ds)\\
&\quad+\int_0^t \int_{|x|>1 } x\, \MU^\Delta(dx,ds)+\mathbf B_t,\qquad t\geq 0.
\end{aligned}\end{equation}
The characteristic triplet depends on the choice of the truncation function $h$ but since $h$ is kept fixed during the proof, we suppress the dependence on $h$ in the notation. From now on we fix the standard truncation function with
\begin{align}
\label{E2.02}
	h(x)=x\1_{\{|x|\leq 1\}}.
\end{align}

\begin{lemma}\label{L2.1}	For every $p\in[1,2)$, $i=1,2$ and $N\in\N$, the processes $\brZ^{N,i}$ are $L^p$-martingales and can be written as two-dimensional stochastic integral equations
\begin{align}\label{E2.03}
\begin{split}
	 \brZ^N_t
	=\brZ^N_0+\sumkS   \int_0^{\beta^N t}\int_0^{I^N_{s-}(k)}\int_{E}\frac{1}{N}J\left(y,X^{N}_{s-}(k)\right)(\mathcal{N-N'})(\{k\},dy,dr,ds).
\end{split}
\end{align}
\end{lemma}
\begin{proof}
The proof of the statement can be found in \cite[pages 540-542]{KM2}. Here we only give a brief outline of the argument.

Note that (\ref{E2.03}) is an immediate consequence of the definition in \eqref{E1.15} and the fact that $\CA^N$ is a $q$-matrix and hence all the drift terms cancel.
In order to show that $Z^{N,i}$ is an $L^p$-martingale, it is enough to show that (for all $p\in[1,2)$, $t>0$ and $k\in S^N$)
\begin{equation}
\label{E2.04}
\E\left[\int_0^{t}I^N_{s-}(k)\int_{E}\left|J\left(y,X^{N}_{s-}(k)\right)\right|^p\,\mathcal{N'}(\{k\},dy,[0,1],ds)\right]<\infty.
\end{equation}
This, however, is a consequence of the fact (which can be checked by a direct computation) that, by definition of $\nu$,
$$\int_E|y-(1,0)|^p\nu(dy)<\infty\quad\text{for all }p\in[1,2).$$
\end{proof}
\begin{proposition}\label{P2.2}
		For the truncation function $h(x)=x\1_{\{|x|\leq 1\}}$, the two-di\-mensional semimartingale $(\brZ_t^N)_{t\geq 0}$ from Theorem \ref{T1} has the characteristic triplet $(\mathbf B^N,\mathbf C^N,\MU^N)$ with
	\begin{align*}
		\mathbf B^N_t&=-\beta^N\sumkS   \int_0^t\int_E \bar h\Big(\frac{1}{N}J\Big(y,\rX^{N}_{s}(k)\Big)\Big)\NU(dy)\,\rI^N_{s}(k)\,ds\\
		&=-\beta^N\sumkS   \int_0^ t\int_E\frac{1}{N}J\left(y,\rX^{N}_{s}(k)\right) \1_{\left\{|J(y,\rX^{N}_{s}(k))|/N> 1\right\}}\NU(dy)\,\rI^N_{s}(k)\,ds,\\
		\mathbf C^N&=0,\\
		\MU^N_t(A)&=\beta^N\sumkS  \int_0^t\int_E \1_{A\setminus\{0\}}\!\left(\frac{1}{N}J\left(y, \rX^N_{s}(k)\right)\right)\NU(dy)\,\rI^N_{s}(k)\,ds,
	\end{align*}
		for $t\geq 0$ and $A\in \mathcal B(\R^2)$.
		\end{proposition}
\begin{proof}
 The drift terms of $X^N$ cancel in the total mass because the migration operator $\mathcal A^N$ is a $q$-matrix. By linearity of the Poissonian integral, we split the integral over $E$ into two integrals containing the small and the large jumps, respectively:
\begin{equation}
\label{E2.05}
\begin{aligned}
		 \brZ^N_t
	&=\brZ^N_0+\sumkS   \int_0^{\beta^N t}\int_0^{I^N_{s-}(k)}\int_{E}\1_{\left\{|J\left(y,X^{N}_{s-}(k)\right)|/N\leq 1\right\}}\\
	 &\qquad\times\frac{1}{N}J\left(y,X^{N}_{s-}(k)\right)(\mathcal{N-N'})(\{k\},dy,dr,ds)\\
	& \quad+\sumkS   \int_0^{\beta^N t}\int_0^{I^N_{s-}(k)}\int_{E}\1_{\left\{|J\left(y,X^{N}_{s-}(k)\right)|/N> 1\right\}}\\
	 &\quad\quad\times\frac{1}{N}J\left(y,X^{N}_{s-}(k)\right)(\mathcal{N-N'})(\{k\},dy,dr,ds).
	\end{aligned}
\end{equation}
By Lemma~\ref{L2.15} below, we have
\begin{equation}
\label{E2.06}
\begin{aligned}
&\quad I^N_{s-}(k)\int_{E}\1_{\left\{|J\left(y,X^{N}_{s-}(k)\right)|/N>1\right\}}\frac{1}{N}\left|J\left(y,X^{N}_{s-}(k)\right)\right|\nu(dy)\\
&\leq I^N_{s-}(k)\frac{8}{N^2}\big[(X^{N,1}_{s-}(k))^2+(X^{N,2}_{s-}(k))^2\big].
\end{aligned}
\end{equation}
Let $D_N<\infty$ be as in Lemma~\ref{L2.3} below. Then, for all $k\in S^N$ and all $s\geq0$, we have
$$I^N_{s-}(k)=\frac{\CA^NX^{N,i}_{s-}(k)}{X^{N,3-i}_{s-}(k)}\leq \frac{D_N}{X^{N,3-i}_{s-}(k)}\mfalls X^{N,3-i}_{s-}(k)>0.$$ Hence the right hand side of (\ref{E2.06}) is bounded by $8D_N^2/N^2<\infty$.

Recall that $\mathcal{N}'(\{k\},dy,dr,ds)=\nu(dy)\,dr\,ds$. Since the right hand side of (\ref{E2.06}) is bounded, the compensator integral of the large jumps in (\ref{E2.05}) is well defined, and we can split the compensated integral of large jumps:
	\begin{align}\label{E2.07}
	\begin{split}
		\brZ^N_t&= \brZ^N_0+\sumkS   \int_0^{\beta^N t}\int_0^{I^N_{s-}(k)}\int_{E}\1_{\left\{|J\left(y,X^{N}_{s-}(k)\right)|/N\leq 1\right\}}\\
&\qquad\times\frac{1}{N}J\left(y,X^{N}_{s-}(k)\right)(\mathcal{N-N'})(\{k\},dy,dr,ds)\\[2mm]
	&\quad +\sumkS   \int_0^{\beta^N t}\int_0^{I^N_{s-}(k)}\int_{E}\1_{\left\{|J\left(y,X^{N}_{s-}(k)\right)|/N>1\right\}}\\
&\qquad\times\frac{1}{N}J\left(y,X^{N}_{s-}(k)\right)\mathcal{N}(\{k\},dy,dr,ds)\\[2mm]
		&\quad -\sumkS   \int_0^{\beta^N t}\int_0^{I^N_{s-}(k)}\int_{E}\1_{\left\{|J\left(y,X^{N}_{s-}(k)\right)|/N>1\right\}}\\
&\qquad\times\frac{1}{N}J\left(y,X^{N}_{s-}(k)\right)\mathcal{N}'(\{k\},dy,dr,ds).
		\end{split}
	\end{align}

Using the definition of $\mathcal N'$, integrating out $dr$ and substituting  $\beta^N$ gives the claim. Note that this final step is also referred to as Grigolionis representation and can be found for jump diffusion processes for instance in Chapter~III.2 of \cite{JacodShiryaev2003}.
\end{proof}

\subsection{Moment Estimates}
\label{S2.3}

	In this section, moment bounds are derived. They will be needed for the tightness proof.

\begin{lemma}
\label{L2.3}
For every $K>0$, we have
$$\P\big[Z^{N,i}_t\geq K\text{ for some }t\geq0\big]\leq \frac{\E[Z^{N,i}_0]}{K},\qquad i=1,2, \, N\in\N.$$
In particular, for all $N$, we have
$$D_N:=\sup_{t\geq0,\,k\in S^N,\, i=1,2}X^{N,i}_t(k)<\infty\quad\text{ a.s.}$$
\end{lemma}
\begin{proof}
Recall that $Z^{N,i}$ is a martingale (Lemma~\ref{L2.1}). Hence the claim is a direct consequence of Doob's inequality.
\end{proof}
	\begin{lemma}\label{L2.4}
		Let $(\bZ^N)_{N\in \N}$ be as in Theorem \ref{T1}, then
		\begin{align*}
			\E\big[Z^{N,1}_tZ^{N,2}_t\big]\leq \E\big[Z^{N,1}_0Z^{N,2}_0\big]
		\end{align*}
		for all $t\geq 0$.
	\end{lemma}
    \begin{proof}
    	Recall from \eqref{E1.03} the transition semigroup $\etAN$ of $\CA^N$. By Lemma 3.7 of \cite{KM2}, we get the mixed second moment bound
    	\begin{align*}
    		\E\big[X^{N,1}_t(k) X^{N,2}_t(l)\big]\leq \E\Big[(\etAN X_0^{N,1})(k) \, (\etAN X_0^{N,2})(l)\Big],\qquad t\geq 0.
    	\end{align*}
As in the previous proof we get
	 \begin{align*}
		\E\big[Z^{N,1}_t& Z^{N,2}_t\big]
		=\E\Big[\frac{1}{N^2} \sum_{k,l=1}^N X^{N,1}_{t}(k) X^{N,2}_{ t}(l)\Big]\\
		&\leq\E\Big[ \frac{1}{N} \sumkS   (\etAN X_0^{N,1})(k)\frac{1}{N} \sum_{l=1}^N(\etAN X_0^{N,2})(l)\Big]
		=\E\big[ Z^{N,1}_0\,Z^{N,2}_0\big]
	\end{align*}	
	for all $t\geq 0$ and $N\in \N$.
	    \end{proof}

\begin{lemma}\label{L2.5}
There exists a constant $C<\infty$ such that (with the convention $0\log0=0$)
\begin{align}\label{E2.08}
	\frac1N\sumkS  \E\big[|X^{N,i}_t(k)\log(X^{N,i}_t(k))|\big]\leq \,C\quad\mbox{for all }  t\geq0,\, N\in\N,\, i=1,2.
\end{align}
\end{lemma}
\begin{proof}
Let $p\in(1,2)$ as in the formulation of Theorem~\ref{T1}. Note that $|x\log(x)|\leq 1$ for $x\in[0,1]$. Furthermore, for $x\geq1$, we have $\log(x)=\frac{1}{p-1}\log(x^{p-1})\leq\frac{x^{p-1}}{p-1}$. Summing up we have
  $$|x\log(x)|\leq 1+\frac{x^p}{p-1}\quad\text{for all }x\geq0.$$
  Hence, the left hand side in (\ref{E2.08}) is bounded by
\begin{equation}
\label{E2.09}
1+\frac{1}{p-1}\frac1N\sumkS  \E\big[(X^{N,i}_t(k))^p\big].
\end{equation}
Recall that $Q_x(dy)$ denotes the harmonic measure of planar Brownian motion on $[0,\infty)^2$ started at $x=(x_1,x_2)\in[0,\infty)^2$. By Theorem 2 of \cite{KlenkeOeler2010}, we get
\begin{equation}
\label{E2.10}
\E\big[X^{N,i}_t(k)^p\big]
=\E\left[\int_E y_i^p\, Q_{(\etAN X^N_0)(k)}(dy)\right].
\end{equation}
By Lemma 2.3 of \cite{KlenkeOeler2010}, we have
\begin{equation}
\label{E2.11}
\begin{aligned}
\int_E y_i^p\, Q_x(dy)&\leq\,\pi\frac{\sin(p/2)}{\sin((\pi/2)p)}\big(x_1^2+x_2^2\big)^{p/2}\qquad\mbox{for }x\in[0,\infty)^2.
\end{aligned}
\end{equation}
Note that $\phi:[1,2]\to\R$, $p\mapsto(2-p)\sin(p/2)/\sin((\pi/2)p)$ is maximal for $p=2$ with $\phi(2)=2\sin(1)/\pi\leq2/\pi$. Further note that $\big(x_1^2+x_2^2\big)^{p/2}\leq x_1^{p}+x_2^{p}$. Concluding, we have
\begin{equation}
\label{E2.12}
\begin{aligned}
\int_E y_i^p\, Q_x(dy)&\leq\,\frac{2}{2-p}\big(x_1^p+x_2^p\big)\qquad\mbox{for }x\in[0,\infty)^2.
\end{aligned}
\end{equation}
Combining (\ref{E2.10}) and (\ref{E2.12}), using Jensen's inequality in the second line and the assumption (\ref{E1.18}) in the fourth line, we obtain
$$\begin{aligned}
\frac1N\sumkS  \E\big[(X^{N,i}_t(k))^p\big]
&\leq \frac{2}{2-p}\frac1N\sumkS  \E\big[((\etAN X^{N,1}_0)(k))^p+((\etAN X^{N,2}_0)(k))^p\big]\\
&\leq \frac{2}{2-p}\sum_{j=1}^2\frac1N\sum_{k,l=1}^N\E\big[\etAN(k,l)(X^{N,j}_0(l))^p\big]\\
&= \frac{2}{2-p}\sum_{j=1}^2\frac1N\sum_{l=1}^N\E\big[(X^{N,j}_0(l))^p\big]\;\leq\; \frac{4\,C_p}{2-p}.
\end{aligned}
$$
Hence the claim holds with $C:=1+\frac{4\,C_p}{(2-p)(p-1)}$.
\end{proof}
\begin{cor}\label{C2.6}
Define
\begin{equation}\label{E2.13}
		Y^{N,i}_t:=\frac{2}{N} \sumkS  \rX^{N,i}_{t}(k)
		\Big(2+\big|\log\big(\rX^{N,i}_{t}(k)\big)\big|\Big),
\end{equation}
then there exists a constant $C<\infty$ such that
\begin{align*}
\E[Y^{N,i}_t]\leq \,C\quad\mbox{for all }  t\geq0,\, N\in\N,\, i=1,2.
\end{align*}
\end{cor}

The next identity will be used frequently. It is a result of the choice of $\mathcal A^N$.
\begin{lemma}\label{L2.7}
For $i=1,2$ we have
\begin{align*}
		\rX^{N,i}_s(k)\rI^{N,3-i}_s(k)=\1_{\{\rX^{N,i}_{s}(k)\neq 0\}}{\mathcal A^N \rX^{N,3-i}_{s}(k)}
		&= \1_{\{\rX^{N,i}_{s}(k)\neq 0\}}\rZ^{N,3-i}_{s}
	\end{align*}
	for any $k=1,...,N$.
\end{lemma}
\begin{proof}
The first equality is immediate from the definition of $\rI^{N,3-i}_s(k)$.

The second equality is a direct consequence of the definition of $I^N$, $\brZ^N$ and $\mathcal A^N$.  In fact, note that $\{\rX^{N,i}_{s}(k)\neq 0\}\subset\{\rX^{N,3-i}_{s}(k)= 0\}$. Hence
\begin{align*}
		\1_{\{\rX^{N,i}_{s}(k)\neq 0\}}{\mathcal A^N \rX^{N,3-i}_{s}(k)}
		&=\1_{\{\rX^{N,i}_{s}(k)\neq0\}}\left(\sum_{l=1}^N \frac{1}{N} \rX^{N,3-i}_{s}(l)- \rX^{N,3-i}_{s}(k)\right)\\
		&=\1_{\{\rX^{N,i}_{s}(k)\neq0\}}\left(\sum_{l=1}^N \frac{1}{N} \rX^{N,3-i}_{s}(l)\right)\\
		&= \1_{\{\rX^{N,i}_{s}(k)\neq 0\}}\rZ^{N,3-i}_{s}.
	\end{align*}
	This proves the second equality.
\end{proof}
The tightness proof requires a subtle choice of the order $p_N$ of the moments to be computed.
\begin{lemma}\label{L2.8}
	Define $p_N=2-\frac{1}{\log N}\in(1,2)$ for $N=3,4,...$. Then
	\begin{align*}
		\frac{\displaystyle N^2}{\displaystyle N^{p_N} \log N}\, \frac{1}{2-p_N}=e.
	\end{align*}
\end{lemma}
\begin{proof}
	Plugging-in.
\end{proof}

	\begin{lemma}\label{L2.9}
		Let $(\brZ^N)_{N\in \N}$ be as in Theorem \ref{T1} and $p_N=2-\frac{1}{\log N}$, then
		\begin{align*}
			\sup_{N\geq 2}\E\Big[\sup_{t\leq T}\big|\rZ^{N,i}_t-\rZ^{N,i}_0\big|^{p_N}\Big]
			\leq 1218\, T\, \E\left[Z_0^{N,1}Z_0^{N,2}+Z_0^{N,1}+Z_0^{N,2}\right],
		\end{align*}
		for $T>0$ and $i=1,2$.
	\end{lemma}
	\begin{proof}
		We only deal with the case $i=1$, the argument for $i=2$ is analogous. Using \eqref{E2.03} gives
		\begin{align*}
			\E\Big[\sup_{t\leq T}\big|\rZ^{N,1}_t-\rZ^{N,1}_0\big|^{p_N}\Big]		
			&= \E\left[\sup_{t\leq T}\left|\frac{1}{N} \sumkS  M^{N,1}_{\beta^Nt}(k)\right|^{p_N}\right]\\
			&= :\frac{1}{N^{p_N}}\E\left[\sup_{t\leq T}\Big|U^{N,1}_{\beta^Nt}\Big|^{p_N}\right]
		\end{align*}		
		with the martingales
		\begin{align*}
			 M^{N,1}_t(k)=\int_0^t\int_0^{I^N_{s-}(k)}\int_{E}J_1(y,X^N_{s-}(k))(\mathcal{N-N'})(\{k\},dy,dr,ds).
		\end{align*}
		The Burkholder-Davis-Gundy inequality (see, e.g., \cite[Thm VII.92]{DellacherieMeyer1983}) applied to the martingale $U^{N,1}$ gives
		 \begin{align}\label{E2.14}
			\frac{1}{N^{p_N}}\E\left[\sup_{t\leq T}\Big|U^{N,1}_{\beta^Nt}\Big|^{p_N}\right]
			\leq \frac{C_{p_N}}{N^{p_N}}\E\left[ \big[U_\ARG^{N,1}, U_\ARG^{N,1}\big]_{\beta^NT}^{p_N/2}\right],
		 \end{align}
		where $C_{p}=(4p)^{p}$ is the Burkholder-Davis-Gundy constant and\linebreak $[U^{N,1}_\ARG, U^{N,1}_\ARG]$ is the quadratic variation of the pure jump martingale $U^{N,1}$. Note that $\sup_{p\in [1,2]} C_p=64<\infty$, so  $C_{p_N}$ is bounded from above by $64$. Next, we need to bound the right hand side of \eqref{E2.14} from above. Note that $\big[U^{N,1}_\ARG, U^{N,1}_\ARG\big]_t$ is the sum of the squared jumps of $U^{N,1}$ up to time $t$. Hence,
		\begin{align}\label{E2.15}
		\begin{split}
			&\quad\big[U^{N,1}_\ARG, U^{N,1}_\ARG\big]^{p_N/2}_{\beta^N t}\\
			&=\Bigg(\sumkS  \int_0^{\beta^N t} \int_0^{I^N_{s-}(k)} \int_E J_1(y, X_{s-}^N(k))^2 \mathcal N(\{k\},dy,dr,ds)\Bigg)^{p_N/2}\\
			&\leq \sumkS  \int_0^{\beta^N t} \int_0^{I^N_{s-}(k)} \int_E \big| J_1(y, X_{s-}^N(k))\big|^{p_N} \mathcal N(\{k\},dy,dr,ds),
			\end{split}
		\end{align}
		where the last inequality follows from the elementary inequality $(\sum_i a_i)^q\leq \sum_i a_i^q$ for all $a_i\geq 0$ and $0<q\leq 1$. In fact, the sum over the triple integral in the second line is an infinite sum with summands $a_i=J_1(y_i,X^N_{s_i-}(k_i))^2$ for certain random points $(y_i,s_i,k_i)$ since we integrate against a Poisson point measure.

Now take expectations on both sides of \eqref{E2.15}, recall that $\mathcal N'$ is the compensator measure of $\mathcal N$, to get
		\begin{align*}
			&\quad\E\Big[\big[U^{N,1}_\ARG, U^{N,1}_\ARG\big]^{p_N/2}_{\beta^N t}\Big]\\
			&\leq \E\left[\sumkS  \int_0^{\beta^N t} \int_0^{I^N_{s-}(k)} \int_E \big|J_1(y, X_{s-}^N(k))\big|^{p_N} \mathcal N'(\{k\},dy,dr,ds)\right].
		\end{align*}
		Applying the definition of $\mathcal N'$, substituting $\beta^N=\frac{N}{\log N}$ in the time-index and plugging in the definition of $J$ gives for \eqref{E2.14} the upper bound
		\begin{align}\label{E2.16}
		\begin{split}
		&\quad \frac{64\,N}{N^{p_N}\log N} \E\left[\sumkS   \int_0^T \int_0^{\rI^N_{s}(k)}\int_E \big|J_1\big(y,\tilde X^N_{s}(k)\big)\big|^{p_N}\mathcal N'(\{k\},dy,dr,ds)\right]\\		
			&=\frac{64\,N^2 }{N^{p_N}\,\log N} \E\bigg[\frac{1}{N}\sumkS   \int_0^{T} \int_0^\infty \big|(y_1-1)\rX^{N,1}_{s}(k)\big|^{p_N}\\
 &\qquad\times\rI^{N,2}_s(k)\, \frac{4}{\pi}\frac{y_1}{(1-y_1)^2(1+y_1)^2}\,dy_1\,ds\bigg]\\
						&+\frac{64\,N^2}{N^{p_N}\,\log N} \E\left[\frac{1}{N}\sumkS   \int_0^{T} \int_0^\infty \big|-\rX^{N,1}_{s}(k)\big|^{p_N}\rI^{N,2}_s(k)\,\frac{4}{\pi}\frac{y_2}{(1+y_2^2)^2}\,dy_2\, ds\right]\\
					& +\frac{64\, N^2}{N^{p_N}\,\log N} \E\left[\frac{1}{N}\sumkS   \int_0^{T} \int_0^\infty \big|y_2\rX^{N,2}_{s}(k)\big|^{p_N}  \rI^{N,1}_s(k)\,\frac{4}{\pi}\frac{y_2}{(1+y_2^2)^2}\, dy_2\,ds\right].\end{split}
					\end{align}
	Let us recall the discussion before the statement of Theorem \ref{T1} to explain the reason for the three cases on the right hand side of the equality: In order to change the first coordinate at some given site $k$ only three of the four types of jumps are being counted:
	\begin{align*}
		\left(x\atop 0\right)\mapsto \left(y_1 x\atop 0\right)\qquad \text{or}\qquad
	\left(x\atop 0\right)\mapsto \left(0\atop  y_2x\right)\qquad \text{or}\qquad
	\left(0\atop x\right)\mapsto \left(y_2 x\atop 0\right)
	\end{align*}
	and these correspond to the three integrals in the same order. To bound the integrands of the summands we first use the trivial bound $a^{p_N-1}\leq 1+a$ and Lemma \ref{L2.7} to get
\begin{equation}
\label{E2.17}
	\begin{aligned}
		&\quad  \frac{1}{N} \sumkS   (\rX^{N,i}_{s}(k))^{p_N}\rI^{N,3-i}_s(k)\\
		&= \frac{1}{N} \sumkS   (\rX^{N,i}_{s}(k))^{p_N-1}\1_{\{\rX^{N,i}_{s}(k)\neq 0\}}{\mathcal A^N \rX^{N,3-i}_{s}(k)}\\
		&\leq \frac{1}{N} \sumkS     (1+\rX^{N,i}_{s}(k)) \1_{\{\rX^{N,i}_{s}(k)\neq 0\}}\rZ^{N,3-i}_{s}\\
		&\leq \rZ^{N,3-i}_s+ \rZ^{N,1}_s\, \rZ^{N,2}_s.
	\end{aligned}
\end{equation}

Using the Fubini-Tonelli theorem and plugging in (\ref{E2.17}) then yields as an upper bound for \eqref{E2.16}:
	\begin{align*}
			&\frac{64\, N^2}{N^{p_N}\,\log N} \bigg(\int_0^\infty \frac{|y_1-1|^{p_N}y_1}{(1-y_1)^2(1+y_1)^2}\,dy_1\\
&\quad+\int_0^\infty\frac{y_2}{(1+y_2^2)^2}\,dy_2 +\int_0^\infty \frac{y_2^{p_N+1}}{(1+y_2^2)^2}\,dy_2\bigg) \\
			&\qquad\times \E\left[\int_0^T \rZ^{N,1}_s\rZ^{N,2}_s\,ds+\int_0^T \rZ^{N,1}_s\,ds+\int_0^T \rZ^{N,2}_s\,ds\right].
				\end{align*}

We compute
$$\int_0^\infty\frac{y_2}{(1+y_2^2)^2}\,dy_2=\frac{\pi}{4}\leq\frac{1}{2-p_N},$$
$$\int_0^\infty\frac{y_2^{p_N+1}}{(1+y_2^2)^2}\,dy_2\leq1+\int_1^\infty y_2^{p_N-3}=1+\frac{1}{2-p_N}\leq\frac{2}{2-p_N}$$
and
$$\begin{aligned}
&\quad\int_0^\infty\frac{|y_1-1|^{p_N}y_1}{(1-y_1)^2(1+y_1)^2}\,dy_1\\
&\leq\int_0^2|1-y_1|^{p_N-2}\,dy_1+\int_1^\infty
\frac{(y_1-1)^{p_N}y_1}{(1-y_1)^2(1+y_1)^2}\,dy_1\\
&\leq \frac{2}{2-p_N}+\int_0^\infty\frac{y_1^{p_N+1}}{(y_1^2+1)^2}\,dy_1\leq \frac{4}{2-p_N}.
\end{aligned}
$$
Summing up, we get as an upper bound for \eqref{E2.16}
	\begin{align*}
	 448\,\frac{N^2}{N^{p_N}\,\log N}\frac{1}{2-p_N}
\E\left[\int_0^T \rZ^{N,1}_s\rZ^{N,2}_s\,ds+\int_0^T \rZ^{N,1}_s\,ds+\int_0^T \rZ^{N,2}_s\,ds\right].
	\end{align*}
Lemma \ref{L2.8} and the choice $p_N=2-\frac{1}{\log N}$ gives
	\begin{align*}
		 \frac{N^2}{N^{p_N}\,\log N}\frac{1}{2-p_N}\equiv e.
	\end{align*}
	Lemmas \ref{L2.1} and \ref{L2.4} then imply the final bound 	
	\begin{align*}
		\E\Big[\sup_{t\leq T}\big|\rZ^{N,1}_t-\rZ^{N,1}_0\big|^{p_N}\Big]	\leq 448\,e\, T\, \E\left[Z_0^{N,1}Z_0^{N,2}+Z_0^{N,1}+Z_0^{N,2}\right]
	\end{align*}
	and the proof is complete.
	\end{proof}

\subsection{Tightness Arguments}
\label{S2.4}
The tightness of $(\brZ^N)_{N\in \N}$ is proved with Aldous's tightness criterion (see Aldous \cite[Theorem 1]{Aldous1978} or \cite[Theorem 6.8]{Walsh1986}.
According to that, in order to prove that  $(\brZ^N)_{N\in \N}$ is tight in the Skorokhod space, it is enough to show that
 \begin{itemize}
   \item[(i)] for every fixed $t\geq 0$, the set of random variables  $(\brZ_t^N)_{N\in \N}$ is tight,
   \item[(ii)] for every sequence of stopping times $(\tau_N)_{N\in\N}$ for the filtrations generated by $(\brZ^N)_{N\in \N}$, bounded above by some $T>0$, and for every sequence of positive real numbers $(\delta_N)_{N\in\N}$ converging to $0$,
               {$|\brZ^{N}_{\tau_N +\delta_N} -  \brZ^{N}_{\tau_N}| \to 0$ in probability as $N\to\infty$.}
 \end{itemize}
We start with (i):
\begin{lemma}\label{L2.10}
	Under the assumptions of Theorem \ref{T1},  the sequence $(\brZ^N_t)_{N\in \N}$ (with values in $[0,\infty)^2$) is tight for any $t>0$.
\end{lemma}
\begin{proof}
From Doob's inequality and Lemma \ref{L2.1} we obtain, for $T>0$ and $K>0$,
\begin{align*}
	\limsup_{N\to \infty} \P\Big[\sup_{t\leq T} Z^{N,i}_t>K\Big]
	\leq \limsup_{N\to\infty}\frac{\E\big[ Z^{N,i}_0\big]}{K}<\infty,\qquad i=1,2.
\end{align*}
Hence, the tightness of $(\brZ^N_t)_{N\in \N}$ follows immediately for any $t\geq 0$.
\end{proof}
Let us next deal with (ii):
\begin{lemma}
\label{L2.11}
	Let $(\brZ^N)_{N\in \N}$ and $C_p$ be as in Theorem \ref{T1} and suppose $(\tau_N)_{N\in\N}$ is a sequence of stopping times for the filtrations generated by $(\brZ^N)_{N\in\N}$, uniformly bounded by some $T>0$. Then, for every $\delta\in(0,1)$ and $N\in\N$ we have
\begin{align*}
\E\big[\big|\rZ^{N,i}_{\tau_N+\delta}-\rZ^{N,i}_{\tau_N}\big|^{3/2}\big]\leq 105\, C_p\,\sqrt{\delta}\qquad i=1,2.
\end{align*}
In particular, if $\delta_N\to 0$, then
\begin{align*}
\big|\rZ^{N,i}_{\tau_N+\delta_N}-\rZ^{N,i}_{\tau_N}\big|\stackrel{P}{\longrightarrow }0,\qquad i=1,2,
\end{align*}
as $N\to\infty$.
\end{lemma}
\begin{proof}
	The lemma is mostly a consequence of the moment bounds and the strong Markov property: Using Lemma \ref{L2.9} and Jensen's inequality for conditional expectations with $p_N=2-\frac{1}{\log N}\in (1,2)$ gives
	\begin{align*}
		 &\quad\E\big[\big|\rZ^{N,i}_{\tau_N+\delta}-\rZ^{N,i}_{\tau_N}\big|^{p_N/2}\big]\\
		&\leq \E\big[\E\big[\big|\rZ^{N,i}_{\tau_N+\delta}-\rZ^{N,i}_{\tau_N}\big|^{p_N}\,\big|\, \mathcal F_{\tau_N}\big]^{1/2}\big]\\
		 &=\E\big[\E_{\rZ^N_{\tau_N}}\big[\big|\rZ^{N,i}_{\delta}-\rZ^{N,i}_{0}\big|^{p_N}\big]^{1/2}\big]\\
		&\leq\sqrt{1218} \sqrt{\delta}\, \E\big[\big(\rZ_{\tau_N}^{N,1}\rZ_{\tau_N}^{N,2}+\rZ_{\tau_N}^{N,1}+\rZ_{\tau_N}^{N,2}\big)^{1/2}\big]\\
		&\leq 35 \sqrt{\delta}\,\left( \E\big[\big(\rZ_{\tau_N}^{N,1}\rZ_{\tau_N}^{N,2}\big)^{1/2}\big]+\E\big[\big(\rZ_{\tau_N}^{N,1}\big)^{1/2}\big]+\E\big[\big(\rZ_{\tau_N}^{N,2}\big)^{1/2}\big]\right).
		\end{align*}
		The last inequality we used the elementary inequality $\sqrt{a+b+c}\leq\sqrt{a}+\sqrt{b}+\sqrt{c}$, $a,b,c\geq0$. Since $\rZ^{N,i}$ is a nonnegative supermartingale, by the optional sampling theorem, we get $\E[\rZ^{N,i}_{\tau_N}]\leq\E[Z_0^{N,i}]$. Hence, also using H\"older's inequality,
				\begin{align*}
			\E\big[\big(\rZ_{\tau_N}^{N,i}\big)^{1/2}\big]&\leq 1+\E\big[\rZ_{\tau_N}^{N,i}\big]\leq 1+\E[Z_0^{N,i}]\leq C_p\\
		\E\big[\big( \rZ^{N,1}_{\tau_N}\rZ^{N,2}_{\tau_N}\big)^{1/2}\big]&\leq \E\big[ \rZ_{\tau_N}^{N,1}\big]^{1/2}\,\E\big[\rZ_{\tau_N}^{N,2}\big]^{1/2}\,\leq \big(\E\big[Z_0^{N,1}\big]\,\E\big[Z_0^{N,2}\big]\big)^{1/2}.
	\end{align*}		
By Markov's inequality and the moment assumption on the initial conditions, the right hand sides of each of the above inequalities are bounded by $C_p$. Hence, the claim follows.
\end{proof}

Next, we prove that the sequence $(\brZ^N)_{N\in \N}$ is $C$-tight, that is, it is tight and all possible limit points are continuous processes. The next proof is also needed in the final step of the proof of Lemma \ref{L2.14} below.
\begin{lemma}\label{L2.12}
	Under the assumptions of Theorem \ref{T1} the sequence $(\brZ^N)_{N\in \N}$ is $C$-tight.
\end{lemma}
\begin{proof}
  By Proposition VI.3.26 (iii) of \cite{JacodShiryaev2003}, we need to show that
      \begin{align}\label{E2.18}
  	\lim_{N\to\infty} \P\Big[\sup_{s\leq t}| \Delta \rZ^{N,i}_s|>\eps\Big]=0
  \end{align}
  for all $t,\eps>0$. According to Lemma VI.4.22 of \cite{JacodShiryaev2003} this can be deduced from
  \begin{align}\label{E2.19}
  	\lim_{N\to\infty} \E\left[\MU^N_t \big(\{|x|>\eps\}\big)\right]=0,
  \end{align}
  with $\MU^N$ from Proposition \ref{P2.2}.
By Corollary A.6 of \cite{KM2}, we have
$$\NU\big(\{y:|J(y,(1,0))|\geq L\}\big)\leq 2L^{-2}\quad\mbox{for all }L>0.$$
Note that
$$J(y,(x_1,0))=x_1J(y,(1,0))\quad\text{ and }\quad J_i(y,(0,x_2))=J_{3-i}((y_2,y_1),(x_2,0))$$ for $x_1\geq0$, $y\in E$ and $i=1,2$. Hence, we infer
$$\NU\big(\{y:|J(y,x)|\geq L\}\big)\leq 2\frac{x_1^2+x_2^2}{L^2}\quad\mbox{for all }x\in E,\,L>0.$$
Hence, using Proposition~\ref{P2.2}, Lemma \ref{L2.7} and Lemma \ref{L2.4},
\begin{align}\label{E2.20}
\begin{aligned}
&\quad\E\big[\MU^N_t \big(\{x:\,|x|>\eps\}\big)\big]\\
	 &=\E\Bigg[\frac{N}{\log N}\sumkS   \int_0^t \int_E \1_{\left\{|J(y,\rX^N_{s}(k))|/N>\eps\right\}} \NU(dy)\,\rI^{N}_s(k)\,ds\Bigg]\hspace*{-1em}\\
	 &\leq\E\Bigg[\frac{2}{\log N}\eps^{-2}\frac{1}{N}\sumkS   \int_0^t \big[\rX^{N,1}(k)^2\rI^{N,2}_s(k)+\rX^{N,2}(k)^2\rI^{N,1}_s(k)\big]\,ds\Bigg]\hspace*{-0.7em}\\
	 &=\frac{4}{\log N}\eps^{-2}\,\int_0^t\E\big[\rZ^{N,1}_s\rZ^{N,2}_s\big]\,ds\\
	 &\leq\frac{4t}{\log N}\eps^{-2}\,\E\big[Z^{N,1}_0Z^{N,2}_0\big]\limN0.
\end{aligned}
\end{align}
\end{proof}

\subsection{Proof of Convergence}\label{S2.5}
To prove convergence  of $(\brZ^N)_{N\in \N}$ to a solution of \eqref{E1.19}, we use general semimartingale theory, see Chapter~IX of \cite{JacodShiryaev2003}.
By Theorem IX.2.4 of \cite{JacodShiryaev2003}, if  $(\mathbf Y^N)_{N\in \N}$ is a sequence of two-dimensional semimartingales with modified characteristics $(\mathbf B^N,  \tilde {\mathbf C}^N,\MU^N)$ and
	\begin{align}\label{E2.21}
	\begin{split}
		\big( \mathbf Y^N,\mathbf B^N,\tilde {\mathbf C}^N\big) &\wlimN(\mathbf Y,\mathbf B,\tilde {\mathbf C}),\\
		\big( \mathbf Y^N, g\ast \MU^N\big)&\wlimN(\mathbf Y,g\ast \MU),
	\end{split}
	\end{align}
	then $\mathbf Y$ is a semimartingale with characteristic triplet $(\mathbf B, \mathbf C, \MU)$. The test functions $g:\R^2\to\R$ are continuous, bounded and vanish in a neighbourhood of the origin (in the terminology of \cite{JacodShiryaev2003} this is the class $C_2(\R^2)\supset C_1(\R^2)$ defined in \cite[VII.2.7]{JacodShiryaev2003}) and the modified characteristic $\tilde {\mathbf C}^N$ is defined as
	\begin{align*}
		\tilde C^{N,i,j}:= C^{N,i,j}+(h_i h_j)\ast \MU^N -\sum_{s\leq \,\ARG} \Delta B_s^{N,i} \Delta B^{N,j}_s,\qquad i,j=1,2.
	\end{align*}
	The convergence results for semimartingales are independent of the choice of the continuous and bounded truncation function $h=\left(h_1\atop h_2\right):\R^2\to \R^2$ which appears in the definition of the characteristic triplet.

	\begin{rem}\label{R2.13}
	The reader may recall that in \eqref{E2.02} we chose the truncation function $h=\left(h_1\atop h_2\right)$ with
	\begin{align}
\label{E2.22}
		h_i(x)=x_i \1_{\{|x|\leq 1\}},\qquad i=1,2.
	\end{align}
Of course, $h$ is not continuous but our proofs are valid nonetheless as explained in the proof of the next lemma.
\end{rem}

In the next lemma we identify the characteristics  of any limiting point of the sequence $(\brZ^N)_{N\in \N}$.
	\begin{lemma}\label{L2.14}
	    If $\bZ=\left(Z^1 \atop Z^2\right)$ is a limiting point of the sequence $(\brZ^N)_{N\in \N}$ from Theorem \ref{T1}, then $\bZ$ is a semimartingale with characteristic triplet
	    \begin{align*}
		\mathbf B=0,\qquad \MU=0\qquad\text{and}\qquad  C^{i,j}_\ARG=\1_{\{i=j\}} \frac{8}{\pi} \int_0^\ARG Z^1_s \, Z^2_s\,ds,\qquad i,j=1,2.
	    \end{align*}

	\end{lemma}
	\begin{proof}	
	    Suppose $\bZ$ is the weak limit of $(\brZ^{N_k})$ for some subsequence $(N_k)$. For ease of notation we replace the subsequence $N_k$ by the entire sequence of natural numbers.\\
	     By Proposition \ref{P2.2}, we get $\tilde C^{N,i,j}=(h_ih_j)\ast \MU^N$ because $\mathbf C^N=0$ and $t\mapsto\mathbf B^N_t$ is continuous (as a sum over integrals over the interval $[0,t]$). The main task in the proof of Lemma \ref{L2.14} is to show that
		\begin{align}
			\big(\brZ^N,\mathbf B^N,h_ih_j\ast \MU^N\big)&\wlimN
            \left(\bZ,0,\1_{\{i=j\}}\frac{8}{\pi}\int_0^\ARG Z^1_sZ^2_s\,ds\right),\qquad i,j=1,2,\label{E2.23}
			\end{align}
			and
			\begin{align}
			\big(\brZ^N,g\ast \MU^N\big)&\wlimN(\bZ,0).\label{E2.24}
		\end{align}
With \eqref{E2.23} and \eqref{E2.24} at hand one would like to apply~\eqref{E2.21} and Theorem IX.2.4 of \cite{JacodShiryaev2003} to finish the proof of the lemma. However we have a technical
issue. In order to apply~\eqref{E2.21} and Theorem IX.2.4 of \cite{JacodShiryaev2003}, one needs the truncation function $h$, which is used in the definition
of characteristics, to be continuous. However, the truncation function $h$ defined in \eqref{E2.22} is discontinuous. Let us show that
 our choice of $h$ suffices to prove the convergence result. Suppose $\tilde h$ is another truncation function such
that $\tilde h(x) = h(x)$ for $|x|\leq 1$, ${\rm supp}(\tilde h)\subset \{x: |x|\leq 2\}$ and $\tilde h$ is bounded and continuous and such that $|\bar {\tilde h}|\leq |\bar h|$, where as before $\bar h(x)=x-h(x)$ and $\bar{\tilde h}(x)=x-\tilde h(x)$. For example, take
$$\tilde h(x)=\begin{cases}
h(x),&\mfalls |x|\leq 1,\\
h(x/|x|)\cdot(2-|x|)^+,&\mfalls |x|\geq 1.
\end{cases}$$
Now denote by $\mathbf B^N(f)$ and $\mathbf C^N(f)$ the modified characteristic with truncation function $f$.
Then
\begin{align*}
	\big|\mathbf B^N_t(\tilde h)\big|
		&=\left|\beta^N\sumkS   \int_0^ t\int_E \bar{\tilde h}\left(\frac{1}{N}J\left(y,\rX^{N}_{s}(k)\right)\right)\NU(dy)\,\rI^N_{s}(k)\,ds\right|\\
		&\leq\beta^N\sumkS   \int_0^ t\int_E \Bigg|\bar{\tilde h}\left(\frac{1}{N}J\left(y,rX^{N}_{s}(k)\right)\right)\Bigg|\,\NU(dy)\,\rI^N_{s}(k)\,ds\\
		&\leq\beta^N\sumkS   \int_0^ t\int_E \Bigg| \bar h\left(\frac{1}{N}J\left(y,rX^{N}_{s}(k)\right)\right)\Bigg|\,\NU(dy)\,\rI^N_{s}(k)\,ds
\end{align*}
and the right hand side is later shown to vanish in the limit, see \eqref{E2.28} and calculations below it.\smallskip

  Also note that by \eqref{E2.19},
\begin{align}
\label{E2.25}
 \E\big[ \tilde h_i\tilde h_j \1_{\{|\,\ARG\,|\in (1,2]\}} \ast \MU^N_t\big] &\leq
\big\| \tilde h_i\big\|_{\infty} \big\|\tilde h_j\big\|_{\infty}
  \E\big[ \MU_t^N(\{|x|\geq 1\})\big]\stackrel{N\to\infty}{\longrightarrow} 0,\qquad t\geq 0,
\end{align}
so that the identity
\begin{align}
\label{E2.26} \tilde C^{N,i,j}_t(\tilde h) &= \tilde h_i\tilde h_j \ast \MU^N_t
  = h_i h_j \ast \MU^N_t+ \tilde h_i\tilde h_j \1_{\{|\,\ARG\,|\in (1,2]\}} \ast \MU^N_t
\end{align}
implies that the pointwise limits of $\tilde C_t^{N,i,j}(\tilde h)$ and $\tilde C_t^{N,i,j}(h)$ coincide. Thus, according to the above and \eqref{E2.21}, the lemma is proved if we can show \eqref{E2.23} and \eqref{E2.24} with $h$ as in~\eqref{E2.22}.\medskip

		Before we start proving \eqref{E2.23} and \eqref{E2.24} we use Skorohod's theorem (Theorem 3.1.8 of \cite{EthierKurtz1986}) to assume in what follows that $(\brZ^N)_{N\in \N}$ converges almost surely in the Skorokhod topology to a limit $\bZ$ and not only weakly. Later in the proof, we will assume this almost sure convergence (instead of convergence in probability) also for two auxiliary processes. Additionally we proved in Lemma \ref{L2.12} that $\bZ$ is a continuous process, thus, $(\brZ^N)_{N\in \N}$ converges to $\bZ$ locally uniformly in time (Proposition VI.1.17 of \cite{JacodShiryaev2003}). But then we also have almost sure convergence of
\begin{align}\label{E2.27}
\lim_{N\to\infty} \int_0^t f(\brZ^{N}_s)\,ds=\int_0^t f(\bZ_s)\,ds<\infty
\end{align}
for any continuous $f:\R_+\times \R_+\to \R$ uniformly for $t\in[0,T]$ for all $T\geq0$.\\

\textbf{\underline{Proof of (\ref{E2.23}).}}\\

Since the limit $Z$ is continuous it suffices to prove separately Skorokhod convergence of the characteristics for each coordinate (Proposition VI.2.2(b) of \cite{JacodShiryaev2003}). The almost sure convergence of $(\brZ^N)_{N\in \N}$ to $\textbf Z$ can be assumed as explained above \eqref{E2.27}; the latter two are proved in what follows.\medskip
			
	 The maps $t\mapsto \frac{8}{\pi}\int_0^t Z^1_s Z^2_s \,ds$ and $t\mapsto 0$ are non-decreasing, hence, in order to prove Skorokhod convergence of the coordinates in (\ref{E2.23}), it is enough to prove the following convergence (Proposition VI.1.17 of \cite{JacodShiryaev2003}): For every $t_0>0$, almost surely, we have uniformly in $t\in[0,t_0]$,
	 \begin{align}
	 		\mathbf B^N_t&\stackrel{N\to\infty}{\longrightarrow}0\label{E2.28}
	\end{align}
	and
	\begin{align}
		h_ih_j\ast \MU^N_t &\stackrel{N\to\infty}{\longrightarrow} \1_{\{i=j\}}\frac{8}{\pi}\int_0^tZ^1_s\,Z^2_s\,ds\label{E2.29}.
	\end{align}
	The most delicate part is \eqref{E2.29} which we prove first. Note that all what follows is based on the almost sure convergence of $(\brZ^N)_{N\in \N}$ so that all convergence statements are in the almost sure sense even if not mentioned explicitly. \smallskip\par
		
	\underline{Verification of \eqref{E2.29}:}	Let $t_0>0$ and $t\in[0,t_0]$. Applying Proposition \ref{P2.2} one finds	
	\begin{align*}
		 h_ih_j\ast \MU^N_t
		&=\beta^N\sumkS  \int_0^{t}\int_E h_i\left(\frac{1}{N}J\left(y,\rX^{N}_{s}(k)\right)\right)\\
&\hspace{7em}\times h_j\left(\frac{1}{N}J\left(y,\rX^{N}_{s}(k)\right)\right)
	 \rI^N_{s}(k)\,\NU(dy)\,ds.		
	\end{align*}
	
Using the definition of $J$ and $\NU$ - compare also with the decomposition in four cases discussed above Theorem \ref{111} or the discussion below \eqref{E2.16} - yields
	\begin{align}\label{E2.30}
	\begin{split}
		&\quad h_ih_j\ast \MU^N_t\\
		&=\beta^N\sumkS  \int_0^{t}\int_0^{\infty}h_i\left(\frac{y_1-1}{N}\rX^{N,1}_{s}(k),0\right)\\
&\hspace{7em}\times h_j\left(\frac{y_1-1}{N}\rX^{N,1}_{s}(k),0\right)\rI^{N,2}_s(k)\,\NU(d(y_1,0))\,ds\\
		&\quad+\beta^N\sumkS  \int_0^{t}\int_0^{\infty}h_i\left(-\frac{1}{N}\rX^{N,1}_{s}(k),\frac{y_2}{N}\rX^{N,1}_{s}(k)\right)\\
&\hspace{7em}\times h_j\left(-\frac{1}{N}\rX^{N,1}_{s}(k),\frac{y_2}{N}\rX^{N,1}_{s}(k)\right)\rI^{N,2}_s(k)\,\NU(d(0,y_2))\,ds\\
		&\quad+\beta^N\sumkS  \int_0^{t}\int_0^{\infty}h_i\left(0,\frac{y_1-1}{N}\rX^{N,2}_{s}(k)\right)\\
&\hspace{7em}\times h_j\left(0,\frac{y_1-1}{N}\rX^{N,2}_{s}(k)\right) \rI^{N,1}_s(k)\,\NU(d(y_1,0))\,ds\\
		&\quad+\beta^N\sumkS  \int_0^{t}\int_0^{\infty}h_i\left(\frac{y_2}{N}\rX^{N,2}_{s}(k),-\frac{1}{N}\rX^{N,2}_{s}(k)\right)\\
&\hspace{7em}\times h_j\left(\frac{y_2}{N}\rX^{N,2}_{s}(k),-\frac{1}{N}\rX^{N,2}_{s}(k)\right) \rI^{N,1}_s(k)\,\NU(d(0,y_2))\,ds\\
		&=:T_t^{N,1}+T^{N,2}_t+T^{N,3}_t+T^{N,4}_t.
		\end{split}
	\end{align}
	In what follows we discuss separately the limit of each summand $T^{N,1}_t,...,T^{N,4}_t$ of \eqref{E2.30} for the cases $i\neq j$ and $i=j$.\\
	
	\textbf{Convergence of \eqref{E2.30} - the cases $i\neq j$:} \\

		First note that the choice of $h$ yields $h_1(x)h_2(x)=x_1x_2 \1_{\{|x|\leq 1\}}$ so that $T^{N,1}_t$ and $T^{N,3}_t$ vanish. Next, we only show that $T^{N,2}_t$ vanishes in the limit, the bounds for $T^{N,4}_t$ are precisely the same exchanging the roles of $X^{N,1}$ and $X^{N,2}$:
	\begin{align*}
		 \big|T^{N,2}_t\big|
		&\leq \beta^N\sumkS  \int_0^{t_0}\int_0^{\infty} \frac{1}{N}
		\rX^{N,1}_{s}(k)\frac{y_2}{N}\rX^{N,1}_{s}(k)\\
&\hspace{9em}\times \1_{\{ y_2\rX^{N,1}_{s}(k)/N\leq 1\}}
		\rI^{N,2}_s(k)\,\NU(d(0,y_2))\,ds.
\end{align*}
	Rearranging terms, Lemma \ref{L2.7}, Lemma \ref{LA.5} and plugging-in the definitions leads to the upper bound
	\begin{align*}
		\big|T^{N,2}_t\big|&\leq \frac{\int y_2
		 \,\NU(dy)}{\log N}\int_0^{t_0}\frac{1}{N}\sumkS   \rX^{N,1}_{s}(k)\rZ^{N,2}_s\,ds=\frac{1}{\log N} \int_0^{t_0} \rZ^{N,1}_s\,\rZ^{N,2}_s\,ds.
	\end{align*}
The right hand side almost surely converges to zero as $N\to\infty$ due to \eqref{E2.27}. This completes the proof of (uniformly in $t\in[0,t_0]$)
\begin{align*}
	\lim_{N\to\infty}  h_ih_j\ast \MU^N_t=0\qquad\mbox{for }i\neq j.
\end{align*}

	\textbf{Bounding \eqref{E2.30} - the cases $i=j$:}\\
	
	  It suffices to discuss $i=j=1$ as the case $i=j=2$ follows from the same calculations by symmetry in $X^{N,1}$ and $X^{N,2}$. We deal with the cases $T^{N,1}_t,...,T^{N,4}_t$ separately.\\

	\textbf{Claim (i): $\lim_{N\to\infty} T^{N,1}_t= \frac{4}{\pi} \int_0^t Z^1_s\,Z^2_s\,ds$.}\\
	
	Lemma \ref{L2.7} gives
\begin{align*}
		T_t^{N,1}&=\beta^N\sumkS  \int_0^t\int_{(0,\infty)\times\{0\}} \frac{(y_1-1)^2}{N^2}\big(\rX^{N,1}_{s}(k)\big)^2\\
&\hspace{10em}\times  \1_{\{|y_1-1|\rX^{N,1}_{s}(k)/N\leq 1\}}
		\rI^{N,2}_s(k)\,\NU\big(dy)\,ds\\
		&\leq \frac{1}{\log N}\int_0^t \rZ^{N,2}_s \frac{1}{N} \sumkS   \rX^{N,1}_{ s}(k)
		 \int_{(1,1+N/\rX^{N,1}_{s}(k))\times\{0\}}(y_1-1)^2\,\NU(dy)\,ds\\
		&\quad +\frac{1}{\log N}\int_0^t \rZ^{N,2}_s\,\frac{1}{N}\sumkS   \rX^{N,1}_{ s}(k)
		  \int_{ (0,1]\times\{0\}}(y_1-1)^2\,\NU(dy)\,ds\\
		&=:T_t^{N,1,1}+T_t^{N,1,2}.
	\end{align*}
By \cite[Lemma A.4]{{KM2}}, for $x>0$, we have
\begin{equation}\label{E2.31}
\int_{(0,x)\times\{0\}}(y_1-1)^2\,\NU(dy)\;=\;
\frac{4}{\pi}\left(\log(1+x)-\frac{x}{1+x}\right).
\end{equation}
Hence $\int_{(0,1]\times\{0\}}(y_1-1)^2\,\NU(dy)\leq 1$ and
$$
		|T^{N,1,2}_t|\leq \frac{1}{\log N}\int_0^{t_0} \rZ^{N,1}_s\,\rZ^{N,2}_s\,ds
	$$
which tends to zero by \eqref{E2.27}.

We now show that $T_t^{N,1,1}\limN\frac{4}{\pi}\int_0^t Z^{1}_s Z^2_sds$.	
By \eqref{E2.31}, we get
\begin{equation}\label{E2.32}\begin{aligned}
&\quad\left|\frac4\pi\log(N)-\int_{(1,1+N/x)\times\{0\}}(y_1-1)^2\,\NU(dy)\right|\\
&=\;
\frac{4}{\pi}\left|\log\Big(\frac2N+\frac1x\Big)-\log(2)-\frac{1+N/x}{2+N/x}+\frac12\right|\;\leq\; 3+2|\log(x)|.
\end{aligned}
\end{equation}
For the last inequality in \eqref{E2.32}, note that by Lemma~\ref{LA.1}
$$\Big|\log\Big(\frac2N+\frac1x\Big)\Big|\leq|\log(x)|+\frac{2}{N}\leq|\log(x)|+1.$$
Furthermore, we have $(1+N/x)/(2+N/x)\in[1/2,1]$ and hence
$$\left|-\log(2)-\frac{1+N/x}{2+N/x}+\frac12\right|\;\leq\;\log(2)+\frac12.$$
Finally, since $\frac4\pi(\log(2)+\frac12+1)\leq 3$, we get the last inequality in \eqref{E2.32}.

Hence we can write
\begin{align*}
		T_t^{N,1,1}&=:\frac{4}{\pi}\int_0^t \rZ^{N,1}_s\rZ^{N,2}_s\,ds+T^{N,1,1,1}_t
\end{align*}
with (recall the definition of $Y^{N,i}_s$ from (\ref{E2.13}))

$$\big|T^{N,1,1,1}_t\big|\leq\frac1{\log N}\int_0^{t_0}\rZ^{N,2}_s\,Y^{N,1}_s\,ds.$$
By Corollary~\ref{C2.6}, there exists a $C<\infty$ such that
$$\E\big[Y^{N,1}_s\big]\leq  C\quad\mbox{for all }s\geq0,\,N\in\N.$$

Let $\varepsilon>0$ be arbitrary. Since $\rZ^{N,2}$ is a nonnegative martingale, Doob's inequality (Lemma~\ref{L2.3}) yields that for $K=K_\varepsilon>0$ large enough and for all $N\in\N$, we have
$$\P\Big[\sup_{s\geq 0}\rZ^{N,2}_s\geq K\Big]\leq\frac1K\E\big[\rZ^{N,2}_0\big]\leq \varepsilon.$$
Define the event
$$A_\varepsilon:=\Big\{\sup_{s\geq0}\rZ^{N,2}_s\leq K\Big\}.$$
Hence
$$\begin{aligned}
\E\big[\big|T^{N,1,1,1}_t\big|\,\1_{A_\varepsilon}\big]&\leq K\cdot\frac1{\log N }\int_0^{t_0}\E\big[Y^{N,1}_s\big]\,ds\,
\leq\, \frac{CKt_0}{\log N}\limN0.
\end{aligned}
$$
Together with $\P[A_\varepsilon^c]<\varepsilon$ this yields that $|T_t^{N,1,1,1}|\limN0$ uniformly in $t\in[0,t_0]$ in probability. Hence the pair $(\brZ^N,T_t^{N,1,1,1})_{N\in\N}$ converges in probability and by Skorokhod's theorem, we can choose a probability space such that the convergence is even almost sure. Putting everything together we proved Claim (i).\\

	\textbf{Claim (ii): $\lim_{N\to\infty} T^{N,2}_t=0$.}\\
	
Lemma \ref{L2.7} gives
	\begin{align*}
	0\leq	T^{N,2}_t&\leq \beta^N\sumkS  \int_0^{t_0} \rZ^{N,2}_s \frac{1}{N^2}\rX^{N,1}_{s}(k)
		\,\int_0^{\infty}\NU(d(0,y_2))\,ds\\
		&= \frac{\NU(\{0\}\times \R^+)}{\log N}\int_0^{t_0} \rZ^{N,1}_s \rZ^{N,2}_s \,ds.
	\end{align*}
	By definition, $\NU(\{0\}\times \R_+)<\infty$, thus,  the right hand side goes to zero by \eqref{E2.27}.\\
	
	\textbf{Claim (iii): $\lim_{N\to\infty} T^{N,3}_t=0$.}\\

	The integral $T^{N,3}_t$ equals zero almost surely for all $t\geq 0$ and $N\in \N$ since the integrand vanishes by definition of $h_i$.\\

	\textbf{Claim (iv): $\lim_{N\to\infty} T^{N,4}_t=\frac{4}{\pi} \int_0^t Z^1_s Z^2_s\,ds$.}\\
	
	To prove the claim we establish an upper bound and a lower bound with the same limit. Lemma \ref{L2.7} and the definition of $\NU$ give
	\begin{align*}
		&0\leq T^{N,4}_t\\
		&\leq\beta^N\sumkS  \int_0^{t_0}\int_0^{\infty}
		 \left(\frac{y_2}{N}\rX^{N,2}_{s}(k)\right)^2\1_{\{y_2\rX^{N,2}_{s}(k)/N\leq 1\}}
	 \rI^{N,1}_s(k)\,\NU(d(0,y_2))\,ds\\
		&=\frac{1}{\log N}\int_0^{t_0} \rZ^{N,1}_s \frac{1}{N}\sumkS   \rX^{N,2}_{s}(k) \int_{\{0\}\times(0,N/\rX^{N,2}_{s}(k))}
		y_2^2\,\NU(dy)\,ds.\\
	\end{align*}
By Lemma~A.2 of \cite{KM2}, for $x>0$, we have
$$\int_{\{0\}\times(0,x)}y^2_2\,\NU(dy)\;=\;\frac2\pi\log(1+x^2)-\frac2\pi\frac{x^2}{1+x^2}.$$
Hence, by Lemma~\ref{LA.1},
$$\begin{aligned}
&\quad \left|\frac{4}{\pi}\log(N)-\int_{\{0\}\times(0,N/x)}y^2_2\,\NU(dy)\right|\\
&\leq\frac{N^2}{x^2+N^2}+\big|\log(1+N^2/x^2)-2\log(N)\big|\\
&\leq1+\big|\log(1/N^2+1/x^2)\big|\\
&\leq\; 1+\frac{1}{N^2}+2|\log(x)|\leq\; 2+2|\log(x)|.\end{aligned}$$
Recall the definition of $Y^{N,i}_s$ from (\ref{E2.13}). We infer
$$T^{N,4}_t=:\frac{4}{\pi} \int_0^t \rZ^{N,1}_s\rZ^{N,2}_s\,ds+ T^{N,4,1}_t$$
with
$$|T^{N,4,1}_t|\leq\frac{1}{\log N}\int_0^{t_0}\rZ^{N,1}_s\,Y^{N,2}_s\,ds.$$
Reasoning as above for $T^{N,1,1,1}_t$, we conclude that $T^{N,4,1}_t\limN0$ uniformly in $t\in[0,t_0]$ a.s. Hence we proved Claim (iv).\medskip\par
	
\underline{Verification of \eqref{E2.28}:} We start with a lemma for bounding the first moments of the jumps.

\begin{lemma}\label{L2.15}
For all $\delta>0$ and $x\in E$, we have
$$\int \delta\,|J(y,x)|\1_{(1,\infty)}(\delta\,|J(y, x)|)\,\NU(dy)\;\leq\; 8\,(x_1^2+x_2^2)\,\delta^2.$$
\end{lemma}
\begin{proof}
By symmetry and linearity of $J$, it is enough to consider the case $\delta=1$ (otherwise replace $\delta x$ by $\tilde x$) and $x=(x_1,0)$. The case $x=(0,x_2)$ is analogous.
Hence
\begin{align*}
	|J(y,x)|\leq|J_1(y,x)|+|J_2(y,x)|=(|y_1-1|+y_2)x_1.
\end{align*}
Note that $0<y_2\leq(1/x_1)-1$ under $\nu(dy)$ implies $y_1=0$ and $|J(y,x)|=x_1\sqrt{y_2^2+1}\leq x_1(y_2+1)\leq 1$, thus,
\begin{align*}
	&\quad x_1^{-1}\int_E |J(y,x)|\1_{(1,\infty)}(|J(y,x)|)\,\NU(dy)\\
	&\leq\int_{|y_1-1|>1/x_1,\,y_2=0}|y_1-1|\NU(dy) + 	 \int_{y_2>(1/x_1)-1, y_1=0}(1+y_2)\,\NU(dy).
\end{align*}
Note that (since the first factor in the second integral to come is bounded by 1)
$$\begin{aligned}
&\quad\int_{|y_1-1|>1/x_1,\,y_2=0}|y_1-1|\NU(dy)\\
&=\frac4\pi\int_{(0,(1-1/x_1)\vee0)\cup(1+1/x_1,\infty)}\frac{|y_1-1|y_1}{(1+y_1)^2}\frac1{(1-y_1)^2}\,dy_1
\leq \frac8\pi\, x_1.\end{aligned}$$
Similarly (note that the first factor in the second integral is easily bounded by 4)
$$\int_{y_2>(1/x_1)-1}(1+y_2)\,\NU(dy)=\frac4\pi\int_{0\vee((1/x_1)-1)}^\infty\frac{y_2(y_2+1)^3}{(y_2^2+1)^2}\frac{1}{(y_2+1)^2}\,dy_2
\leq \frac{16}{\pi}\, x_1.$$
The claim follows since $\frac{8+16}\pi\leq8$.
\end{proof}

Let $t_0>0$ and $t\in[0,t_0]$. By Propsition \ref{P2.2}, Lemma~\ref{L2.7} and Lemma~\ref{L2.15} with $\delta=1/N$, we get \eqref{E2.28} since
$$\begin{aligned}
\label{www}
\big|\mathbf B^N_t(h)\big|
&\leq\beta^N\sumkS   \int_0^{t_0}\!\int_E \frac{1}{N}\big|J\big(y,\rX^{N}_{s}(k)\big)\big|\1_{\{|J(y, \rX^N_{s})|/N\geq1\}}\NU(dy)\,\rI^N_{s}(k)\,ds\\
&\leq\frac{8}{\log N}\int_0^{t_0}\frac1N\sumkS  \big[(\rX^{N,1}_s(k))^2\rI^{N,2}_s(k)+(\rX^{N,2}_s(k))^2\rI^{N,1}_s)(k)\big]\,ds\\
&=\frac{16}{\log N}\int_0^{t_0} \rZ^{N,1}_s\rZ^{N,2}_s\,ds\limN0.
\end{aligned}
$$

\textbf{\underline{Proof of \eqref{E2.24}.}}\smallskip\par

	By assumption, there are $\eps,c>0$ so that $g(\,\ARG\,)\leq c \1_{\{|\,\ARG\,|>\eps\}}$. For the indicator the bounds were already derived in the proof of Lemma \ref{L2.12}, and hence we are done.
\end{proof}

Now we are ready to finish the proof of Theorem \ref{T1}:
\begin{proof}[Proof of Theorem \ref{T1}]
	The previous Lemma~\ref{L2.14} identifies the semi\-mart\-in\-gale triplet of any possible limit point of the tight sequence $(\brZ^N)_{N\in \N}$.  Chapter~III.2c of \cite{JacodShiryaev2003} (more precisely Theorem~III.2.32) shows that any limit point $ \bZ=(Z^1, Z^2)$ is a weak solution to the two-dimensional stochastic differential equation \eqref{E1.19} started in $\bZ_0=z=\lim_{N\to \infty}\left (Z^{N,1}_0, Z^{N,2}_0\right)$.

Let $\varepsilon>0$ and define
$$\tau_\varepsilon:=\inf\big\{t:(Z^1_t,Z^2_t)\not\in[\varepsilon,\infty)^2\big\}.$$
Note that pathwise uniqueness holds for the SDE \eqref{E1.22} for $t\leq\tau_\varepsilon$ since the noise coefficient is Lipschitz. Letting $\varepsilon\downarrow0$, we get pathwise uniqueness up to time $\tau:=\sup_{\varepsilon>0}\tau_\varepsilon$. Furthermore, we have $Z^i_t=Z^i_\tau$ for $t\geq\tau$, $i=1,2$, since the noise term vanishes as soon as $Z^1_t=$ or $Z^2_t=0$. Hence, pathwise uniqueness holds for all $t\geq0$. Thus all limit points of $(\brZ^N)_{N\in \N}$ are identical in distribution. But this proves weak convergence of $(\brZ^N)_{N\in \N}$ to the unique solution of \eqref{E1.22}.	
\end{proof}
\smallskip
\section{Proof of Theorem~\ref{T2}}\label{S3}

We will need special classes of convergence determining functions on $\R_+^2$ and on $E$, respectively.
For $x=(x_1,x_2)$ and $y=(y_1,y_2)\in\R^2$, define the \emph{lozenge product}
\begin{equation}
\label{E3.01}
x \mtimes{} y\;:=\;-(x_1+x_2)(y_1+y_2)\;+\;i(x_1-x_2)(y_1-y_2)
\end{equation}
(with $i=\sqrt{-1}$) and set
\begin{equation}
\label{E3.02}
F(x,\,y)=\exp(x\mtimes y).
\end{equation}
Note that $x\mtimes y=y\mtimes x$. By \cite[Corollary 2.4]{KlenkeMytnik2010}, the functions $\{F(\ARG,z), z\in\R_+^2\}$ and $\{F(\ARG,z), z\in E\}$ are measure and (weak) convergence determining on $\R_+^2$ and on $E$, respectively. Note that for $y\in E$, the function $F(\ARG,y)$ is harmonic so that, for $\theta\in\R_+^2$ and $y\in E$,
\begin{equation}\label{E3.03}
\int_EQ_\theta(dx)F(x,y)=F(\theta,y).
\end{equation}

We start with a lemma that states the approximate duality relation.

\begin{lemma}\label{L3.1}
Let $ s,t\geq 0$, $n\in\N$, $N\geq n$ and $k_1,\ldots,k_n\in S^N$. If $y(j)=(y_1(j),y_2(j))\in E$, then, for every $\theta\in\R_+^2$, we have
\begin{equation}
\label{E3.05}
\begin{aligned}
&\quad\E\left[\prod_{j=1}^nF\big(X^N_{t+s}(k_j), y(j)\big)\Big|\CF_t\right]\\
&=\prod_{j=1}^nF(\theta, y(j))\prod_{j=1}^nF\big((X^N_t(k_j)-\theta),e^{-s}y(j)\big)\\
&\quad+\E\left[\int_0^s\left(\prod_{j=1}^nF\big(X^N_{t+r}(k_j), e^{r-s}y(j)\big)F\big(\theta,(1-e^{r-s})y(j)\big)\right)\right.\\
&\qquad\qquad\times
\left.\left(e^{r-s}(\bZ^N_{t+r}-\theta)\mtimes\sum_{j=1}^ny(j)\right)dr\;\bigg|\;\CF_t\right].
\end{aligned}
\end{equation}
\end{lemma}
Before we prove the lemma, we show how it implies Theorem~\ref{T2}.

\begin{proof}[Proof of Theorem~\ref{T2}(\rm{i})]
Because $F$ is convergence determining, a simple application of the Cramer-Wold device shows that in order to prove Theorem~\ref{T2}(i), it is enough to show that \begin{equation}
\label{E3.04}
\lim_{N\to\infty}\E\left[F\big(\bZ^N_{\beta^Nt},y\big)\prod_{j=1}^nF\big(X^N_{\beta^Nt}(k_j),y(j)\big)\right]
=\E\left[F\big(\bZ_{t},y\big)\prod_{j=1}^nF\big(\bZ_{t}, y(j)\big)\right]
\end{equation}
for all $t>0$, $n\in\N$, $k(1),\ldots,k(n)\in\N$, $y(1),\ldots,y(n)\in E$ and $y\in\R_+^2$. Let $t>0$ and recall that $\beta^N=N/\log N$. Define $u_N:=2\log N$. We use Lemma~\ref{L3.1} with $t$ replaced by $\beta^Nt-u_N$ and $s$ replaced by $u_N$. Furthermore, we assume $\theta=\bZ^N_{\beta^Nt-u_N}$.
Note that \begin{equation}
\label{E3.06}
\begin{aligned}
\E\big[\big|1-F\big(X^N_{\beta^Nt-u_N}(k_j), e^{-u_N}y(j)\big)\big|\big]&\leq e^{-u_N}|y(j)|N\E\big[\big|\bZ^N_{\beta^Nt-u_N}\big|\big]\\
&\leq e^{-u_N}\,N\,|y(j)|\,\E\big[\big|\bZ^N_0\big|\big]\\
&\leq C/N
\end{aligned}
\end{equation}
and that the first factor in the integral in \eqref{E3.05} is bounded by 1.

Since $\brZ^N$ is $C$-tight and since $u_N/\beta^N\limN0$, we have $\bZ^N_{\beta^Nt-u_N}-\bZ^N_{\beta^Nt}\limN0$ in probability. Hence, for $k_1,\ldots,k_n\in S^N$, $y(j)=(y_1(j),y_2(j))\in E$, $y(0)\in\R_+^2$ and $\theta\in\R_+^2$, we have
\begin{equation}
\label{E3.07}
\begin{aligned}
&\quad\lim_{N\to\infty}\Bigg|\E\Bigg[F\big(\bZ^N_{\beta^Nt}, y(0)\big)\prod_{j=1}^nF\big(X^N_{\beta^Nt}(k_j), y(j)\big)\Bigg]\\
&\qquad-\E\Bigg[F\big(\bZ_{t}, y(0)\big)\prod_{j=1}^nF\big(\bZ_{t}, y(j)\big)\Bigg]\Bigg|\\
&=\lim_{N\to\infty}\Bigg|\E\Bigg[F\big(\bZ^N_{\beta^Nt-u_N}, y(0)\big)\prod_{j=1}^nF\big(X^N_{\beta^Nt}(k_j), y(j)\big)\Bigg]\\
&\quad-\E\Bigg[F\big(\bZ^N_{\beta^Nt-u_N}, y(0)\big)\prod_{j=1}^nF\big(\bZ^N_{\beta^Nt-u_N}, y(j)\big)\Bigg]\Bigg|\\
&\leq\int_0^{u_N}
\E\Bigg[\Bigg|e^{r-u_N}\big(\bZ^N_{\beta^Nt-u_N+r}-\bZ^N_{\beta^Nt-u_N}\big)\mtimes\sum_{j=0}^ny(j)\Bigg|\Bigg]dr\\
&\leq C\sup_{r\in[0,u_N]}\E\big[\big|\bZ^N_{\beta^Nt-u_N+r}-\bZ^N_{\beta^Nt-u_N}\big|\big]
\end{aligned}
\end{equation}
By Lemma~\ref{L2.9}, the random variables
$$\sup_{r\in[0,u_N]}\big|\bZ^N_{\beta^Nt-u_N+r}-\bZ^N_{\beta^Nt-u_N}\big|,\qquad N\geq 2,$$
are $L^q$-bounded for some $q>1$ and are hence uniformly integrable. Since they converge to $0$ in probability, the dominated convergence theorem yields
$$\lim_{N\to\infty} \sup_{r\in[0,u_N]}\E\big[\big|\bZ^N_{\beta^Nt-u_N+r}-\bZ^N_{\beta^Nt-u_N}\big|\big]=0.$$
This finishes the proof of Theorem~\ref{T2}(i).
\end{proof}
\begin{proof}[Proof of Theorem~\ref{T2}(\rm{ii})]
The convergence of finite dimensional distributions is derived by standard methods and we only sketch the main idea.

Let $\theta\in\R_+^2$ and let $Y^\theta$ be the stationary process with distribution $\check Q_\theta$  (recall that it was defined after~\eqref{E1.21}).
Let $m\in\N$ and $s_1<s_2<\ldots<s_m$ as well as $y_1,\ldots,y_m\in E$. Recall from \eqref{E1.21} that
$$
\P[Y^\theta_{s_k}\in dy'|Y^\theta_{s_{k-1}}=y]=Q_{e^{-(s_k-s_{k-1})}y+(1-e^{-(s_k-s_{k-1})})\theta}(dy'),
$$
and hence, using also \eqref{E3.03}, we get (for $z_k\in\R_+^2$)
$$
\begin{aligned}
&\quad\E\big[F(Y^\theta_{s_k},z_k)|Y^\theta_s,\,s\leq s_{k-1}\big]\\
&=\int_E Q_{z_k}(dz)\E\big[F(Y^\theta_{s_k},z)|Y^\theta_s,\,s\leq s_{k-1}\big]\\
&=\int_E Q_{z_k}(dz)F\big(Y^\theta_{s_{k-1}},e^{-(s_k-s_{k-1})}z\big)\,F\big(\theta,(1-e^{-(s_k-s_{k-1})})z\big).
\end{aligned}$$
Iterating the argument, we get
\begin{equation}\label{E3.08}
\begin{aligned}
\E\Bigg[\prod_{k=1}^mF\big(Y^\theta_{s_k},z_k\big)\Bigg]
=G_m(\theta,z_1,\ldots,z_m),\qquad z_1,\ldots,z_m\in\R_+^2,
\end{aligned}
\end{equation}
where the functions $G_k$ are defined iteratively by
$$G_1(\theta,z_1)=\int_EQ_{z_1}(dz)F(\theta,z)$$
and
\begin{equation}
\label{E3.09}
\begin{aligned}
&\quad G_k(\theta,z_1,\ldots,z_k)\\
&=\int Q_{z_k}(dz)G_{k-1}(\theta,z_1,\ldots,z_{k-2},z_{k-1}+e^{s_{k-1}-s_{k}}z)
\,F(\theta,(1-e^{s_{k-1}-s_k})z).
\end{aligned}
\end{equation}
In particular, for $y_1,y_2\in E$, we have
\begin{equation}
\label{E3.10}
\begin{aligned}
G_2(\theta,y_1,y_2)
=F(\theta,(1-e^{s_{1}-s_2})y_2)\int_EQ_{y_1+e^{s_{1}-s_2}y_2}(dz)F(\theta,z).
\end{aligned}
\end{equation}

In order to show Theorem~\ref{T2}(ii), we have to show: For $n\in\N$ and $y(j,k)\in E$, $k=1,\ldots,m$, $j=1,\ldots,n$, $y(0)\in\R_+^2$, and $k_1,\ldots,k_n\in \N$ we have

\begin{equation}
\label{E3.11}
\begin{aligned}
&\quad \lim_{N\to\infty}\E\left[F\big(\bZ^N_{\beta^Nt}, y(0)\big)\prod_{j=1}^n\prod_{k=1}^mF\big(X^N_{\beta^Nt+s_k}(k_j), y(j,k)\big)\right]\\
&=\E\left[F(\bZ_t,y(0))\prod_{j=1}^nG_m\big(\bZ_t,y(j,1),\ldots,y(j,m)\big)\right]\end{aligned}
\end{equation}

For ease of notation, we restrict ourselves to showing \eqref{E3.11} for $m=2$ only.
Using Lemma~\ref{L3.1}, and arguing as in the proof of Theorem~\ref{T2}(i),  we get

\begin{equation}
\label{E3.12}
\begin{aligned}
&\quad \lim_{N\to\infty}\E\left[F\big(\bZ^N_{\beta^Nt}, y(0)\big)\prod_{j=1}^nF\big(X^N_{\beta^Nt+s_2}(k_j), y(j,2)\big)\prod_{j=1}^nF\big(X^N_{\beta^Nt+s_1}(k_j), y(j,1)\big)\right]\\
&=\lim_{N\to\infty}\E\Bigg[F\big(\bZ^N_{\beta^Nt-u_N}, y(0)\big)\prod_{j=1}^nF\big(\bZ^N_{\beta^Nt-u_N}, (1-e^{s_1-s_2})y(j,2)\big)\\
&\qquad \times F\big(X^N_{\beta^Nt+s_1}(k_j), (e^{s_1-s_2}y(j,2)+y(j,1))\big)\Bigg]\\
&=\lim_{N\to\infty}\E\Bigg[F\big(\bZ^N_{\beta^Nt-u_N}, y(0)\big)\prod_{j=1}^nF\big(\bZ^N_{\beta^Nt-u_N}, (1-e^{s_1-s_2})y(j,2)\big)\\
&\qquad\times\int_EQ_{(e^{s_1-s_2}y(j,2)+y(j,1))}(dz)F\big(X^N_{\beta^Nt+s_1}(k_j), z\big)\Bigg]\\
&=\lim_{N\to\infty}\E\Bigg[F\big(\bZ^N_{\beta^Nt-u_N}, y(0)\big)\prod_{j=1}^nF\big(\bZ^N_{\beta^Nt-u_N}, (1-e^{s_1-s_2})y(j,2)\big)\\
&\qquad\times\int_EQ_{(e^{s_1-s_2}y(j,2)+y(j,1))}(dz)F\big(\bZ^N_{\beta^Nt-u_N}, z\big)\Bigg]\\
&=\E\Bigg[F\big(\bZ_{t}, y(0)\big)\prod_{j=1}^nF\big(\bZ_{t}, (1-e^{s_1-s_2})y(j,2)\big)\\
&\qquad\times\int_EQ_{(e^{s_1-s_2}y(j,2)+y(j,1))}(dz)F(\bZ_t, z)\Bigg]\\
&=\E\Bigg[F(\bZ_{t}, y(0))\prod_{j=1}^nG_2(\bZ_{t}, y(j,1),y(j,2))\Bigg]
\end{aligned}
\end{equation}

\end{proof}
\begin{proof}[Proof of Lemma~\ref{L3.1}]
By \cite[Theorem 1.1]{KM2}, we have that
$$
\begin{aligned}
M_s:=&\prod_{j=1}^nF\big(X^N_{t+s}(j),y(j)\big)-\prod_{j=1}^nF\big(X^N_t(j), y(j)\big)\\
&-\int_0^s\prod_{j=1}^nF\big(X^N_{t+r}(j), y(j)\big)\sum_{j=1}^n\big(\bZ^N_{t+r}-X^N_{t+r}(j)\big)\mtimes y(j)\,dr
\end{aligned}
$$
is a martingale with $M_0=0$.
Now we replace $y(j)$ by the time dependent function $e^{u+s}y(j)$ for some $u\in\R$ and subtract the resulting drift to get
$$
\begin{aligned}
M_s':=
&\prod_{j=1}^nF\big(X^N_{t+s}(j), e^{u+s}y(j)\big)F\big(\theta, (1-e^{u+s})y(j)\big)\\
&-\prod_{j=1}^nF\big(X^N_t(j), e^{u}y(j)\big)F\big(\theta, (1-e^{u})y(j)\big)\\
&-\int_0^s\prod_{j=1}^nF\big(X^N_{t+r}(j), e^{u+r}y(j)\big)F\big(\theta, (1-e^{u+r})y(j)\big)\\
&\quad\times\sum_{j=1}^n\big(\bZ^N_{t+r}-X^N_{t+r}(j)\big)\mtimes (e^{u+r}y(j))\,dr\\
&-\int_0^s\prod_{j=1}^nF\big(X^N_{t+r}(j), e^{u+r}y(j)\big)F\big(\theta, (1-e^{u+r})y(j)\big)\\
&\quad\times\sum_{j=1}^n\big(X^N_{t+r}(j)-\theta\big)\mtimes (e^{u+r}y(j))\,dr\\
=&\prod_{j=1}^nF\big(X^N_{t+s}(j), e^{u+s}y(j)\big)F\big(\theta, (1-e^{u+s})y(j)\big)\\
&-\prod_{j=1}^nF\big(X^N_t(j), e^{u}y(j)\big)F\big(\theta, (1-e^{u})y(j)\big)\\
&-\int_0^s\prod_{j=1}^nF\big(X^N_{t+r}(j), e^{u+r}y(j)\big)F\big(\theta, (1-e^{u+r})y(j)\big)\\
&\quad\times\sum_{j=1}^n\big(\bZ^N_{t+r}-\theta\big)\mtimes (e^{u+r}y(j))\,dr
\end{aligned}
$$
is a martingale with $M_0'=0$. Choosing $u=-s$ and taking conditional expectations gives the claim.
\end{proof}

\smallskip
\section*{Appendix A}
\setcounter{lemma}{0}
\setcounter{equation}{0}
\renewcommand{\thelemma}{A.\arabic{lemma}}
\renewcommand{\theequation}{A.\arabic{equation}}
\begin{lemma}\label{LA.1}
Let $a>0$. For all $x>0$ and $y\in[0,a]$, we have
$$|\log(x+y)|\leq a+|\log(x)|.$$
\end{lemma}
\begin{proof}
If $x+y<1$, then $|\log(x+y)|\leq|\log(x)|$. If $x+y\geq1$, then by Taylor expansion, we get
$$|\log(x+y)|=\log(x+y)\leq\log(x+a)\leq \log(x)+a\leq|\log(x)|+a.$$
\end{proof}

We collect some basic properties of the measure $\nu$ defined in \eqref{E1.14}.
\begin{lemma}
\label{LA.2}
Let $\varepsilon>0$. We have
$$
\nu\big(\{0\}\times(\varepsilon,\infty)\big)\;=\;\frac{2}{\pi}\frac{1}{(1+\varepsilon)^2}\;\leq\; \frac{2}{\pi}\big(1\wedge \varepsilon^{-2}\big),
$$
and
$$
\begin{aligned}
\nu\big(([0,\infty)\setminus(1-\varepsilon,1+\varepsilon))\times \{0\}\big)&=
\begin{cases}
\frac{8}{\pi}\frac{1}{\varepsilon(4-\varepsilon^2)}-\frac2\pi,&\mfalls\varepsilon\leq 1,\\[2mm]
\frac{2}{\pi}\frac{1}{\varepsilon(2+\varepsilon)},&\mfalls \varepsilon\geq 1,
\end{cases}\\
&\leq\frac{2}{\pi}\big(\varepsilon^{-1}\wedge\varepsilon^{-2}\big).
\end{aligned}
$$
\end{lemma}
\begin{proof}This is simple calculus.\end{proof}
\begin{lemma}
\label{LA.3}
For $x\geq0$, we have
and
$$\int_{\{0\}\times(0,x)}y^2_2\,\nu(dy)\;=\;\frac2\pi\log(1+x^2)-\frac2\pi\frac{x^2}{1+x^2}\;\leq\;\frac4\pi\,\log(x),$$
where the inequality holds if $x\geq2$.
\end{lemma}
\begin{proof}This is simple calculus.\end{proof}

\begin{lemma}
\label{LA.4}
For $x>0$, we have
$$
\int_{(0,x)\times\{0\}}(y_1-1)^2\,\nu(dy)\;=\;
\frac{4}{\pi}\left(\log(1+x)-\frac{x}{1+x}\right).
$$
Hence, for $\varepsilon\in(0,1)$, we get
$$\int_{(1-\varepsilon,1+\varepsilon)\times\{0\}}(y_1-1)^2\,\nu(dy)
=\frac{4}{\pi}\left(\log\left(\frac{2+\varepsilon}{2-\varepsilon}\right)-\frac{2\varepsilon}{4-\varepsilon^2}\right)
\;\leq\;\frac{2}{\pi}\,\varepsilon$$
\end{lemma}
\begin{proof}This is simple calculus.\end{proof}
\begin{lemma}
\label{LA.5}
We have
$$
\int_E y_2\,\nu(dy)\;=\;1
$$
\end{lemma}
\begin{proof}This is simple calculus.\end{proof}
\textbf{Acknowledgement:} The authors thank J\'{e}r\^{o}me Blauth for proof reading an early draft of the article.	
\newcommand{\etalchar}[1]{$^{#1}$}
\def\cprime{$'$}

\end{document}